\documentclass{louloupart}
\usepackage{pinlabel}
\usepackage{graphicx}
\title[Virtual Artin groups]{Virtual Artin groups}
\author[P Bellingeri]{Paolo Bellingeri}
\givenname{Paolo}
\surname{Bellingeri}
\address{Paolo Bellingeri, Normandie Univ, UNICAEN, CNRS, LMNO, 14000 Caen, France}
\email{paolo.bellingeri@unicaen.fr}
\author[L Paris]{Luis Paris}
\givenname{Luis}
\surname{Paris}
\address{Luis Paris, IMB, UMR 5584, CNRS, Universit\'e Bourgogne Franche-Comt\'e, 21000 Dijon, France}
\email{lparis@u-bourgogne.fr}
\author[A-L Thiel]{Anne-Laure Thiel}
\givenname{Anne-Laure}
\surname{Thiel}
\address{Anne-Laure Thiel, Normandie Univ, UNICAEN, CNRS, LMNO, 14000 Caen, France}
\email{anne-laure.thiel@unicaen.fr}

\newtheorem{thm}{Theorem}[section]
\newtheorem{lem}[thm]{Lemma}
\newtheorem{prop}[thm]{Proposition}
\newtheorem{corl}[thm]{Corollary}
\newtheorem{conj}[thm]{Conjecture}

\theoremstyle{definition}
\newtheorem*{defin}{Definition}
\newtheorem*{rem}{Remark}
\newtheorem*{expl}{Example}
\newtheorem*{expl1}{Example 1}
\newtheorem*{expl2}{Example 2}
\newtheorem*{acknow}{Acknowledgments}

\numberwithin{equation}{section}

\makeatletter
\renewcommand{\thefigure}{\ifnum \c@section>\z@ \thesection.\fi
 \@arabic\c@figure}
\@addtoreset{figure}{section}
\makeatother

\begin{document}

\def\VB{{\rm VB}} \def\N{\mathbb N} \def\Prod{{\rm Prod}}
\def\SS{\mathcal S} \def\TT{\mathcal T} \def\VA{{\rm VA}}
\def\KVA{{\rm KVA}} \def\PVA{{\rm PVA}} \def\PVB{{\rm PVB}}
\def\KVB{{\rm KVB}} \def\PP{\mathcal P} \def\fin{{\rm fin}}
\def\Ker{{\rm Ker}} \def\id{{\rm id}} \def\R{\mathbb R}
\def\GL{{\rm GL}} \def\Z{\mathbb Z} \def\XX{\mathcal X}
\def\YY{\mathcal Y} \def\sph{{\rm sph}} \def\rank{{\rm rank}}
\def\K{\mathbb K} \def\MM{\mathcal M} \def\cd{{\rm cd}}
\def\vcd{{\rm vcd}}


\begin{abstract}
Starting from the observation that the standard presentation of a virtual braid group mixes the standard presentation of the corresponding braid group with the standard presentation of the corresponding symmetric group and some mixed relations that mimic the action of the symmetric group on its root system, we define a virtual Artin group $\VA[\Gamma]$ of a Coxeter graph $\Gamma$ mixing the standard presentation of the Artin group $A[\Gamma]$ with the standard presentation of the Coxeter group $W[\Gamma]$ and some mixed relations that mimic the action of $W[\Gamma]$ on its root system.
By definition we have two epimorphisms $\pi_K:\VA[\Gamma]\to W[\Gamma]$ and $\pi_P:\VA[\Gamma]\to W[\Gamma]$ whose kernels  are denoted by $\KVA[\Gamma]$ and $\PVA[\Gamma]$ respectively. 
We calculate presentations for these two subgroups.
In particular $\KVA[\Gamma]$ is an Artin group.
We prove that the center of any virtual Artin group is trivial.
In the case where $\Gamma$ is of spherical type or of affine type, we show that each free of infinity parabolic subgroup of $\KVA[\Gamma]$ is also of spherical type or of affine type, and we show that $\VA[\Gamma]$ has a solution to the word problem. 
In the case where $\Gamma$ is of spherical type we show that $\KVA[\Gamma]$ satisfies the $K(\pi,1)$ conjecture and we infer the cohomological dimension of $\KVA[\Gamma]$ and the virtual cohomological dimension of $\VA[\Gamma]$.
In the case where $\Gamma$ is of affine type we determine upper bounds for the cohomological dimension of $\KVA[\Gamma]$ and for the virtual cohomological dimension of $\VA[\Gamma]$.
\end{abstract}

\maketitle


\section{Introduction}\label{sec1}

Virtual braid groups were introduced by Kauffman \cite{Kauff1} in his founding paper on the theory of virtual knots and links. 
They play the same role for virtual links as the braid groups for classical links.
We know by Kamada \cite{Kamad1} and Vershinin \cite{Versh1} that the virtual braid group on $n$ strands, denoted $\VB_n$, admits a presentation with generators $\sigma_1,\dots,\sigma_{n-1},\tau_1,\dots,\tau_{n-1}$, and relations:
\begin{itemize}
\item[(v1)]
$\sigma_i\sigma_j=\sigma_j\sigma_i$ for $|i-j|\ge2$, \ $\sigma_i\sigma_j\sigma_i=\sigma_j\sigma_i\sigma_j$ for $|i-j|=1$,
\item[(v2)]
$\tau_i\tau_j=\tau_j\tau_i$ for $|i-j|\ge2$, \ $\tau_i\tau_j\tau_i=\tau_j\tau_i\tau_j$ for $|i-j|=1$, \ $\tau_i^2=1$ for $1\le i\le n-1$,
\item[(v3)]
$\tau_j\sigma_i=\sigma_i\tau_j$ for $|i-j|\ge2$, \ $\tau_i\tau_j\sigma_i=\sigma_j\tau_i\tau_j$ for $|i-j|=1$.
\end{itemize}
The relations (v1) are the relations of Artin's presentation for the braid group $B_n$, and they determine an embedding of $B_n$ into $\VB_n$. 
The relations (v2) are the relations of the Coxeter type presentation of the symmetric group $S_n$, and they determine an embedding of $S_n$ into $\VB_n$.
As for the relations (v3), called \emph{mixed relations}, they mimic the action of the symmetric group on its associated root system.
This situation extends naturally to all Artin groups and Coxeter groups, which explains the forthcoming definition of virtual Artin group.

A \emph{Coxeter matrix} on a countable set $S$ is a symmetric square matrix $M=(m_{s,t})_{s,t\in S}$, indexed by the elements of $S$, with entries in $\N\cup\{\infty\}$, satisfying $m_{s,s}=1$ for all $s\in S$ and $m_{s,t}=m_{t,s}\ge2$ for all $s,t\in S$, $s\neq t$.
Such a matrix is usually represented by a labeled graph $\Gamma$, called \emph{Coxeter graph}, defined as follows.
The set $S$ is the set of vertices of $\Gamma$.
Two vertices $s,t$ are joined by an edge if $m_{s,t}\ge3$, and this edge is labeled with $m_{s,t}$ if $m_{s,t}\ge 4$.
We say that $\Gamma$ (or $M$) is \emph{finite} if $S$ is finite.
In most of the studies on Coxeter groups and/or Artin groups the Coxeter graph is assumed to be finite.
Here we will not assume that $\Gamma$ is necessarily finite since, from a given Coxeter graph $\Gamma$, we will construct another Coxeter graph $\hat\Gamma$ which is not finite in general, even when $\Gamma$ is.
However, we will always assume the set of vertices to be countable.

If $a,b$ are two letters and $m$ is an integer $\ge2$, then we denote by $\Prod_L(a,b,m)$ the word $aba\cdots$ of length $m$ and by $\Prod_R(b,a,m)$ the word $\cdots aba$ of length $m$.
If $a,b$ are two elements of a group $G$, then we denote again by $\Prod_L(a,b,m)$ (resp. $\Prod_R(b,a,m)$) the element of $G$ represented by $\Prod_L(a,b,m)$ (resp. $\Prod_R(b,a,m)$).
Let $\Gamma$ be a Coxeter graph and let $M=(m_{s,t})_{s,t\in S}$ be its Coxeter matrix. 
The \emph{Artin group} associated with $\Gamma$, denoted $A[\Gamma]$, is the group defined by the following presentation: 
\[
A[\Gamma]=\langle S\mid \Prod_R(t,s,m_{s,t}) = \Prod_R(s,t,m_{s,t})\text{ for }s,t\in S\,,\ s\neq t\,,\ m_{s,t}\neq\infty\rangle\,.
\]
The \emph{Coxeter group} associated with $\Gamma$, denoted $W[\Gamma]$, is the quotient of $A[\Gamma]$ by the relations $s^2=1$, $s\in S$.

\begin{defin}\label{defVA}
Let $\Gamma$ be a Coxeter graph and let $M=(m_{s,t})_{s,t\in S}$ be its Coxeter matrix. 
Let $\SS=\{\sigma_s\mid s\in S\}$ and $\TT=\{\tau_s\mid s\in S\}$ be two abstract sets in one-to-one correspondence with $S$.
The \emph{virtual Artin group} associated with $\Gamma$, denoted $\VA[\Gamma]$, is the group defined by the presentation with generating set $\SS\sqcup\TT$ and relations:
\begin{itemize}
\item[(v1)]
$\Prod_R(\sigma_t,\sigma_s,m_{s,t})=\Prod_R(\sigma_s,\sigma_t,m_{s,t})$ for $s,t\in S$, $s\neq t$ and $m_{s,t}\neq\infty$;
\item[(v2)]
$\Prod_R(\tau_t,\tau_s,m_{s,t})=\Prod_R(\tau_s,\tau_t,m_{s,t})$ for $s,t\in S$, $s\neq t$ and $m_{s,t}\neq\infty$; $\tau_s^2=1$ for $s\in S$;
\item[(v3)]
$\Prod_R(\tau_s,\tau_t,m_{s,t}-1)\,\sigma_s=\sigma_r\,\Prod_R(\tau_s,\tau_t,m_{s,t}-1)$, where $r=s$ if $m_{s,t}$ is even and $r=t$ if $m_{s,t}$ is odd, for $s,t\in S$, $s\neq t$ and $m_{s,t}\neq\infty$.
\end{itemize}
\end{defin}

Like for virtual braid groups, the relations (v1) are the relations that define the Artin group $A[\Gamma]$, and we have a natural homomorphism $\iota_A:A[\Gamma]\to\VA[\Gamma]$ which sends $s$ to $\sigma_s$ for all $s\in S$.
On the other hand, the relations (v2) are the relations that define the Coxeter group $W[\Gamma]$, and we have a natural homomorphism $\iota_W:W[\Gamma]\to\VA[\Gamma]$ which sends $s$ to $\tau_s$ for all $s\in S$.
We will see in Section \ref{sec2} that these two homomorphisms are injective. 
Finally, the relations (v3) mimic the action of $W[\Gamma]$ on its root system (see Lemma \ref{lem2_2}). 

Artin groups were introduced in the early 1970s by Brieskorn \cite{Bries1}, Brieskorn--Saito \cite{BriSai1} and Deligne \cite{Delig1} as natural generalizations of braid groups. 
This generalization goes beyond the presentations themselves since a large part of the theory of braid groups extends to some Artin groups, especially those of spherical type, and a large part of the study of braid groups is now made in this more general framework.
We expect that the same phenomenon will hold for virtual Artin groups viewed as generalizations of virtual braid groups.

Notice that the third author \cite{Thiel1}  proposed a notion of virtual braid group of type $B_n$ by extending   the diagrammatic description provided by tom Dieck for  Artin groups of type $B$
and which turns out to be  a quotient of  the virtual Artin group of type $B_n$ defined above.

The virtual braid group on $n$ strands $\VB_n$ admits two surjective homomorphisms onto the symmetric group $S_n$, $\pi_K:\VB_n\to S_n$ and $\pi_P:\VB_n\to S_n$.
Actually, when $n\ge 5$, by Bellingeri--Paris \cite{BelPar1}, these are the only surjective homomorphisms from $\VB_n$ onto $S_n$ up to automorphism of $S_n$.
The kernel of $\pi_P$ is called the \emph{pure virtual braid group} and it is denoted by $\PVB_n$, and the kernel of $\pi_K$ is called the \emph{kure virtual braid group} and it is denoted by $\KVB_n$.
The homomorphisms $\pi_K$ and $\pi_P$ admit a common section $\iota_n:S_n\to \VB_n$, hence we have the semi-direct product decompositions $\VB_n=\KVB_n\rtimes S_n$ and $\VB_n=\PVB_n\rtimes S_n$.
A presentation for $\PVB_n$ was calculated by Bardakov \cite{Barda1}.
Furthermore, $\PVB_n$ coincides with a group studied by Bartholdi--Enriquez--Etingof--Rains \cite{BaEnEtRa1} in relation with the Yang--Baxter equations.
By Bartholdi--Enriquez--Etingof--Rains \cite{BaEnEtRa1}, $\PVB_n$ admits a finite classifying space, that is, $\PVB_n$ is of type F, which implies also that it is torsion free.
On the other hand, the group $\KVB_n$ is an Artin group (see Rabenda \cite{Raben1} and Bardakov--Bellingeri \cite{BarBel1}) and it is quite well-understood thanks to Godelle--Paris \cite{GodPar2}.
In particular $\KVB_n$ is also of type F, therefore it is torsion free and it has a solution to the word problem.
Since $\KVB_n$ is of finite index in $\VB_n$, we can therefore use $\KVB_n$ to solve the word problem in $\VB_n$ (see Bellingeri--Cisneros-de-la-Cruz--Paris \cite{BeCiPa1}).

Let $\Gamma$ be a Coxeter graph.
Like for virtual braid groups, we have two surjective homomorphisms $\pi_P:\VA[\Gamma]\to W[\Gamma]$ and $\pi_K:\VA[\Gamma]\to W[\Gamma]$ that admit a common section $\iota_W:W[\Gamma]\to\VA[\Gamma]$.
Let $\PVA[\Gamma]$ denote the kernel of $\pi_P$ and $\KVA[\Gamma]$ denote the kernel of $\pi_K$.
Then we have semi-direct product decompositions $\VA[\Gamma]=\PVA[\Gamma]\rtimes W[\Gamma]$ and $\VA[\Gamma]=\KVA[\Gamma]\rtimes W[\Gamma]$.
In a first step we compute presentations for $\KVA[\Gamma]$ and $\PVA[\Gamma]$ (see Theorem \ref{thm2_3} and Theorem \ref{thm2_6}).
These presentations are generalizations of the presentations of $\PVB_n$ and $\KVB_n$ found by Bardakov \cite{Barda1}, Rabenda \cite{Raben1} and Bardakov--Bellingeri \cite{BarBel1}, and provide a new insight into these groups. 
For example $\KVA[\Gamma]$ as well as $\PVA[\Gamma]$ has a generating set naturally in one-to-one correspondence with the root system of $W[\Gamma]$, and the action of $W[\Gamma]$ on $\KVA[\Gamma]$ and/or $\PVA[\Gamma]$ is given by the action of $W[\Gamma]$ on its root system. 
As a first application we show that the natural homomorphism $\iota_A: A[\Gamma]\to\VA[\Gamma]$ which sends $s$ to $\sigma_s$ for all $s\in S$ is injective (see Corollary \ref{corl2_4}).
Note that the homomorphism $\iota_W:W[\Gamma]\to\VA[\Gamma]$ is necessarily injective since it is a section of both $\pi_P$ and $\pi_K$ (see Proposition~\ref{prop2_1}).

From Theorem \ref{thm2_3} it follows that $\KVA[\Gamma]$ is an Artin group whose Coxeter graph, denoted $\hat\Gamma$, is determined by the root system of $W[\Gamma]$.
We use this in the rest of the paper. 
In Section \ref{sec3} we prove that the center of $\VA[\Gamma]$ is trivial whatever $\Gamma$ is (see Corollary \ref{corl3_4}).
This result is rather surprising since determining the center of all Artin groups is an open question.
In Sections \ref{sec4}, \ref{sec5} and \ref{sec6} we assume that the Coxeter graph $\Gamma$ is of spherical type or of affine type. 
There is a general principle which says that, if one understands the free of infinity parabolic subgroups of a given Artin group, then we can understand the Artin group itself (see Godelle--Paris \cite{GodPar1}).
The aim of Section \ref{sec4} is to show that we can apply this principle to $\KVA[\Gamma]$ when $\Gamma$ is of spherical type or of affine type (see Theorem \ref{thm4_1}). 
In particular we show that each free of infinity parabolic subgroup of $\KVA[\Gamma]$ is of spherical type or of affine type and its spherical dimension is bounded above by the spherical dimension of $A[\Gamma]$.
The definition of ``spherical dimension'' of an Artin group is given in the beginning of Section \ref{sec4}.  
In the good cases, that include Artin groups of spherical type and of affine type, it coincides with the cohomological dimension of the group. 
In Section \ref{sec5} we show that, if $\Gamma$ is a Coxeter graph of spherical type or of affine type, then $\VA[\Gamma]$ has a solution to the word problem (see Theorem \ref{thm5_1}).
Section \ref{sec6} concerns some (co)homological properties of $\VA[\Gamma]$ and $\KVA[\Gamma]$ when $\Gamma$ is of spherical type or of affine type. 
We show that, if $\Gamma$ is of spherical type, then $\KVA[\Gamma]$ satisfies the so-called $K(\pi,1)$ conjecture (see Theorem \ref{thm6_3}), hence it is of type F, its cohomological dimension is finite, and it is torsion free (see Corollary \ref{corl6_5}).
It also follows that the virtual cohomological dimension of $\VA[\Gamma]$ is equal to the number of vertices of $\Gamma$.
For the case where $\Gamma$ is of affine type we show that the cohomological dimension of $\KVA[\Gamma]$ and the virtual cohomological dimension of $\VA[\Gamma]$ are bounded above (see Theorem \ref{thm6_6}).

In order to avoid possible misunderstandings let us finish this section with a caveat. The term "virtual" has different meanings in this text:  a virtual Artin group is called in such way because
it is a natural extension of the virtual braid group $\VB_n$, while "virtual cohomogical dimension" and "virtually torsion free" are referred to combinatorial properties of finite index subgroups
of virtual Artin groups.  However, when $\Gamma$ is of spherical type, the fact that   $\KVA[\Gamma]$ is an Artin group and that  $\VA[\Gamma]=\KVA[\Gamma]\rtimes W[\Gamma]$
can be rephrased saying  that the  virtual Artin group  $\VA[\Gamma]$  is  virtually an Artin group.

\begin{acknow}
We thank Conchita Mart\'inez-P\'erez for all the information and references on cohomological dimension of groups that she gave us.
Her help was important for achieving the results of Section \ref{sec6}.
The first and second authors are supported by the French project ``AlMaRe'' (ANR-19-CE40-0001-01) of the ANR. The first  and third authors were also supported by the project ARTIQ (ERDF/RIN). 
\end{acknow}


\section{The groups $\KVA[\Gamma]$ and $\PVA[\Gamma]$}\label{sec2}

Let $\Gamma$ be a Coxeter graph and let $M=(m_{s,t})_{s,t\in S}$ be its Coxeter matrix.
Recall that $\Gamma$ is not assumed to be finite in the sense that the set of vertices of $\Gamma$ can be a countable infinite set.
For $X\subset S$ we set $M_X=(m_{s,t})_{s,t\in X}$ and we denote by $\Gamma_X$ the Coxeter graph of $M_X$.
Note that $\Gamma_X$ is the full Coxeter subgraph of $\Gamma$ spanned by $X$.
We denote by $\PP_\fin(S)$ the set of finite subsets of $S$.
If $X,Y\in\PP_\fin(S)$ are such that $X\subset Y$, then, by Van der Lek \cite[Lemma 4.11]{Lek1} and Bourbaki \cite[Corollaire 1, Page 19]{Bourb1}, the homomorphisms $A[\Gamma_X]\to A[\Gamma_Y]$ and $W[\Gamma_X]\to W[\Gamma_Y]$ induced by the inclusion are injective.
It follows that $A[\Gamma]$ and $W[\Gamma]$ are the following direct limits: 
\[
A[\Gamma]=\varinjlim_{X\in\PP_\fin(S)}A[\Gamma_X]\text{ and }W[\Gamma]=\varinjlim_{X\in\PP_\fin(S)}W[\Gamma_X]\,.
\]
So, most of the results on Artin and Coxeter groups that we will use in this paper, although proved within the framework of finite Coxeter graphs, extend immediately to Coxeter graphs with countable many vertices.
We will point out when and where this is not the case.

We see in the presentations of $\VA[\Gamma]$ and $W[\Gamma]$ that there are surjective homomorphisms $\pi_K:\VA[\Gamma]\to W[\Gamma]$ and $\pi_P:\VA[\Gamma]\to W[\Gamma]$ defined by:
\begin{gather*}
\pi_K(\sigma_s)=1\text{ and }\pi_K(\tau_s)=s\text{ for all }s\in S\,,\\
\pi_P(\sigma_s)=\pi_P(\tau_s)=s\text{ for all }s\in S\,.
\end{gather*}
We set $\KVA[\Gamma]=\Ker(\pi_K)$ and $\PVA[\Gamma]=\Ker(\pi_P)$.
Recall the homomorphism $\iota_W:W[\Gamma]\to\VA[\Gamma]$ which sends $s$ to $\tau_s$ for all $s\in S$.
Observe that $\pi_K\circ\iota_W=\id$ and $\pi_P\circ\iota_W=\id$, hence $\iota_W$ is a section of both epimorphisms.
So:

\begin{prop}\label{prop2_1}
Let $\Gamma$ be a Coxeter graph.
The homomorphism $\iota_W:W[\Gamma]\to\VA[\Gamma]$ is injective and we have the following semi-direct product decompositions:
\[
\VA[\Gamma]=\KVA[\Gamma]\rtimes W[\Gamma]\text{ and }\VA[\Gamma]=\PVA[\Gamma]\rtimes W[\Gamma]\,.
\]
\end{prop}

Let $\Gamma$ be a Coxeter graph and let $M=(m_{s,t})_{s,t\in S}$ be its Coxeter matrix.
Let $\Pi=\{\alpha_s\mid s\in S\}$ be an abstract set in one-to-one correspondence with $S$.
The elements of $\Pi$ are called \emph{simple roots}.
Let $V$ be the real vector space having $\Pi$ as a basis and let $\langle\cdot,\cdot\rangle:V\times V\to\R$ be the symmetric bilinear form, called the \emph{canonical bilinear form}, defined by:
\[
\langle \alpha_s,\alpha_t\rangle=\left\{\begin{array}{ll}
-2\cos(\pi/m_{s,t})&\text{if }m_{s,t}\neq\infty\,,\\
-2&\text{if }m_{s,t}=\infty\,.
\end{array}\right.
\]
There is a faithful linear representation $W[\Gamma]\hookrightarrow\GL(V)$, called the \emph{canonical linear representation} of $W[\Gamma]$, preserving the bilinear form $\langle\cdot,\cdot\rangle$, and defined by:
\[
s(v)=v-\langle v,\alpha_s\rangle\alpha_s\,,\quad \text{for }v\in V\,,\ s\in S\,.
\]
We will assume $W[\Gamma]$ to be embedded into $\GL(V)$ via this linear representation.
We set $\Phi[\Gamma]=\{w(\alpha_s)\mid s\in S\text{ and }w\in W[\Gamma]\}$, and we denote by $\Phi^+[\Gamma]$ (resp. $\Phi^-[\Gamma]$) the set of elements $\beta\in\Phi[\Gamma]$ that can be written $\beta=\sum_{s\in S}\lambda_s\alpha_s$ with $\lambda_s\ge0$ (resp. $\lambda_s\le0$) for all $s\in S$.
By Deodhar \cite{Deodh1} we have the disjoint union $\Phi[\Gamma]=\Phi^+[\Gamma]\sqcup\Phi^-[\Gamma]$ and $\Phi^-[\Gamma]=\{-\beta\mid\beta\in\Phi^+[\Gamma]\}$.
The set $\Phi[\Gamma]$ is called the \emph{root system} of $\Gamma$ and the elements of $\Phi^+[\Gamma]$ (resp. $\Phi^-[\Gamma]$) are called \emph{positive roots} (resp. \emph{negative roots}).

For each $\beta\in\Phi[\Gamma]$ we denote by $r_\beta:V\to V$ the linear reflection defined by $r_\beta(v)=v-\langle v,\beta\rangle\beta$.
Note that, if $s\in S$ and $w\in W[\Gamma]$ are such that $\beta=w(\alpha_s)$, then $r_\beta=wsw^{-1}$.
So $r_\beta\in W[\Gamma]$ for all $\beta\in\Phi[\Gamma]$ and $r_{\alpha_s}=s$ for all $s\in S$.
Note also that $r_\beta=r_{-\beta}$ for all $\beta\in\Phi[\Gamma]$.

We denote by $\cdot$ the action of $W[\Gamma]$ on $\KVA[\Gamma]$ (resp. $\PVA[\Gamma]$).
In other words, for $w\in W[\Gamma]$ and $g\in\KVA[\Gamma]$ (resp. $g\in\PVA[\Gamma]$), we set $w\cdot g=\iota_W(w)\,g\,\iota_W(w)^{-1}$.
The following lemma gives its real meaning to the remark ``the relations (v3) mimic the action of $W[\Gamma]$ on the root system of $W[\Gamma]$'' made in the introduction.

\begin{lem}\label{lem2_2}
Let $\Gamma$ be a Coxeter graph and let $M=(m_{s,t})_{s,t\in S}$ be its Coxeter matrix.
Let $u,v\in W[\Gamma]$ and $s,t\in S$.
If $u(\alpha_s)=v(\alpha_t)$, then $u\cdot\sigma_s=v\cdot\sigma_t$ (in $\KVA[\Gamma]$) and $u\cdot(\tau_s\sigma_s)=v\cdot(\tau_t\sigma_t)$ (in $\PVA[\Gamma]$).
\end{lem}

\begin{proof}
We prove the equality $u\cdot\sigma_s=v\cdot\sigma_t$.
The other equality can be proved in the same way.
We denote by $\lg:W[\Gamma]\to\N$ the word length with respect to $S$.
The proof is based on the following result which can be found in Deodhar \cite{Deodh1}.

{\it Claim 1.}
\begin{itemize}
\item[(1)]
Let $w\in W[\Gamma]$ and $s\in S$.
We have $w(\alpha_s)\in\Phi^+[\Gamma]$ if and only if $\lg(ws)=\lg(w)+1$.
\item[(2)]
Let $w\in W[\Gamma]$ and $s,t\in S$, $s\neq t$, such that $\lg(ws)=\lg(wt)=\lg(w)-1$.
Then $m_{s,t}\neq\infty$ and $\lg(w\,\Prod_R(t,s,m_{s,t})^{-1})=\lg(w)-m_{s,t}$.
\item[(3)]
Let $s,t\in S$, $s\neq t$, such that $m_{s,t}\neq\infty$.
Then $\Prod_R(s,t,m_{s,t}-1)(\alpha_s)=\alpha_r$, where $r=s$ if $m_{s,t}$ is even and $r=t$ if $m_{s,t}$ is odd.
\end{itemize}

Let $u,v\in W[\Gamma]$ and $s,t\in S$ such that $u(\alpha_s)=v(\alpha_t)$.
Upon replacing $u$ with $v^{-1}u$, we may assume that $v=1$, that is, $u(\alpha_s)=\alpha_t$.
We show that $u\cdot\sigma_s=\sigma_t$ by induction on $\lg(u)$.
If $\lg(u)=0$, then $\alpha_s=\alpha_t$, hence $s=t$, and therefore $\sigma_s=\sigma_t$.
So, we can assume that $\lg(u)\ge1$ and that the induction hypothesis holds. 
Since $u(\alpha_s)=\alpha_t$, we have $usu^{-1}=t$, hence $us=tu$.
Moreover, $\lg(us)=\lg(u)+1$, since $u(\alpha_s)=\alpha_t\in\Phi^+[\Gamma]$ (by Claim 1).
Set $w=us=tu$.
Choose $r\in S$ such that $\lg(ur)<\lg(u)$.
We have $r\neq s$, since $\lg(us)=\lg(u)+1$ and $\lg(ur)=\lg(u)-1$.
Moreover, $\lg(ws)=\lg(wr)=\lg(w)-1$ hence, by Claim 1, $m_{r,s}\neq\infty$ and $\lg(w') = \lg(w)-m_{r,s}$, where $w'=w\,\Prod_R(s,r,m_{r,s})^{-1}$.
We set $x=s$ if $m_{r,s}$ is even and $x=r$ if $m_{r,s}$ is odd.
Since $w=us=w'\,\Prod_R(s,r,m_ {r, s} -1)\,s$, we have $u=w'\,\Prod_R(s,r,m_{r,s}-1)$, hence, by Claim 1,
\[
\alpha_t = u(\alpha_s)=w'\,\Prod_R(s,r,m_{r,s}-1)(\alpha_s)=w'(\alpha_x)\,.
\]
Since $\lg(w')=\lg(w)-m_{r,s}\le\lg(w)-2<\lg(w)-1=\lg(u)$, by the induction hypothesis, $w'\cdot\sigma_x=\sigma_t$.
On the other hand,
\begin{gather*}
\Prod_R(s,r,m_{r,s}-1)\cdot\sigma_s=
\Prod_R(\tau_s,\tau_r,m_{r,s}-1)\,\sigma_s\,\Prod_R(\tau_s,\tau_r,m_{r,s}-1)^{-1}=\\
\sigma_x\,\Prod_R(\tau_s,\tau_r,m_{r,s}-1)\,\Prod_R(\tau_s,\tau_r,m_{r,s}-1)^{-1}=
\sigma_x\,,
\end{gather*}
hence:
\[
u\cdot\sigma_s=w'\,\Prod_R(s,r,m_{r,s}-1)\cdot\sigma_s=w'\cdot\sigma_x=\sigma_t\,.
\proved
\]
\end{proof}

Let $\Gamma$ be a Coxeter graph and let $M=(m_{s,t})_{s,t\in S}$ be its Coxeter matrix, we now define some important elements of our two kernels. Let $\beta\in\Phi[\Gamma]$, we choose $w\in W[\Gamma]$ and $s\in S$ such that $w(\alpha_s)=\beta$ and we set $\delta_\beta=w\cdot\sigma_s\in\KVA[\Gamma]$ and $\zeta_\beta=w\cdot(\tau_s\sigma_s)\in\PVA[\Gamma]$.
By Lemma \ref{lem2_2} these definitions do not depend on the choice of $w$ and $s$.
We will see in the proofs of Theorem \ref{thm2_3} and Theorem \ref{thm2_6} that $\{\delta_\beta\mid\beta\in\Phi[\Gamma]\}$ generates $\KVA[\Gamma]$ and $\{\zeta_\beta\mid\beta\in\Phi[\Gamma]\}$ generates $\PVA[\Gamma]$.

We define a new Coxeter matrix $\hat M=(\hat m_{\beta,\gamma})_{\beta,\gamma\in\Phi[\Gamma]}$ indexed by the elements of $\Phi[\Gamma]$ as follows.
\begin{itemize}
\item[(a)]
We set $\hat m_{\beta,\beta}=1$ for all $\beta\in\Phi[\Gamma]$.
\item[(b)]
Let $\beta,\gamma\in\Phi[\Gamma]$, $\beta\neq\gamma$.
If there exist $w\in W[\Gamma]$ and $s,t\in S$ such that $\beta=w(\alpha_s)$, $\gamma=w(\alpha_t)$ and $m_{s,t}\neq\infty$, then we set $\hat m_{\beta,\gamma}=m_{s,t}$.
We set $\hat m_{\beta,\gamma}=\infty$ otherwise.
\end{itemize}
Note that the definition of $\hat m_{\beta,\gamma}$ in Case (b) does not depend on the choice of $w$, $s$ and $t$.
Indeed, suppose there exist $w,w'\in W[\Gamma]$ and $s,t,s',t'\in S$ such that $m_{s,t}\neq\infty$, $m_{s',t'}\neq\infty$, $\beta=w(\alpha_s)=w'(\alpha_{s'})$ and $\gamma=w(\alpha_t)=w'(\alpha_{t'})$.
Then:
\[
\langle\beta,\gamma\rangle=-2\cos(\pi/m_{s,t})=-2\cos(\pi/m_{s',t'})\,,
\]
hence $m_{s,t}=m_{s',t'}$.
We denote by $\hat\Gamma$ the Coxeter graph of $\hat M$.
Moreover, for convenience, we will denote by $\{\hat\delta_\beta\mid\beta\in\Phi[\Gamma]\}$ the standard generating set of $A[\hat\Gamma]$.

\begin{rem}
We know by Deodhar \cite{Deodh1} that $\Phi[\Gamma]$ is finite if and only if $W[\Gamma]$ is finite.
So, $\hat\Gamma$ is finite if and only if $W[\Gamma]$ is finite.
\end{rem}

\begin{expl1}
Suppose $\Gamma$ is the graph $A_{n-1}$ of Figure \ref{fig2_1}.
Then $W[\Gamma]$ is the symmetric group $S_n$, $A[\Gamma]$ is the braid group $B_n$ and $\VA[\Gamma]$ is the virtual braid group $\VB_n$.
We set $U=\R^n$ that we assume to be endowed with the standard scalar product, $\langle\cdot,\cdot\rangle$, and we denote by $\{e_1,\dots,e_n\}$ its canonical basis.
We consider the action of $S_n$ on $U$ by permutations of the coordinates and we denote by $V$ the hyperplane of $U$ defined by the equation $x_1+\cdots+x_n=0$.
Then $V$ is invariant under the action of $S_n$ and the induced representation $W[\Gamma]=S_n\to\GL(V)$ is the canonical representation.
For $i,j\in\{1,\dots,n\}$, $i\neq j$, we set $\alpha_{i,j}=e_j-e_i$.
Then $\Pi=\{\alpha_{1,2},\dots,\alpha_{n-1,n}\}$ is the set of simple roots, and $\Phi[\Gamma]=\{\alpha_{i,j}\mid 1\le i\neq j\le n\}$ is the root system of $W[\Gamma]=S_n$.
The set $\{\delta_{\alpha_{i,j}}\mid 1\le i\neq j\le n\}$ is the generating set of $\KVB_n=\KVA[\Gamma]$ given by Rabenda \cite{Raben1} and Bardakov--Bellingeri \cite{BarBel1} and the set $\{\zeta_{\alpha_{i,j}}\mid 1\le i\neq j\le n\}$ is the generating set of $\PVB_n=\PVA[\Gamma]$ given by Bardakov \cite{Barda1}.
The Coxeter matrix $\hat M$ is defined by:
\begin{itemize}
\item
$\hat m_{\alpha_{i,j},\alpha_{i,j}}=1$ for $1\le i\neq j\le n$.
\item
$\hat m_{\alpha_{i,j},\alpha_{k,\ell}}=2$ for $i,j,k,\ell$ pairwise distinct.
\item
$\hat m_{\alpha_{i,j},\alpha_{j,k}}=3$ for $i,j,k$ pairwise distinct.
\item
$\hat m_{\alpha_{i,j},\alpha_{k,\ell}}=\infty$ otherwise.
\end{itemize}
\end{expl1}

\begin{figure}[ht!]
\begin{center}
\includegraphics[width=4cm]{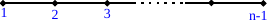}
\caption{Coxeter graph $A_{n-1}$}\label{fig2_1}
\end{center}
\end{figure}

\begin{expl2}
We say that $A[\Gamma]$ (or $\Gamma$) is \emph{simply laced} if $m_{s,t}\in\{2,3\}$ for all $s,t\in S$, $s\neq t$.
Suppose $\Gamma$ is simply laced.
Then each root $\beta\in\Phi[\Gamma]$ can be written $\beta=\sum_{s\in S}\lambda_s\alpha_s$ with $\lambda_s\in\Z$ for all $s\in S$.
Moreover, we have $\langle\beta,\gamma\rangle\in\Z$ for all $\beta,\gamma\in\Phi[\Gamma]$.
In this case $\hat M$ is defined by:
\begin{itemize}
\item
$\hat m_{\beta,\beta}=1$ for $\beta\in\Phi[\Gamma]$.
\item
$\hat m_{\beta,\gamma}=2$ for $\beta,\gamma\in\Phi[\Gamma]$, $\beta\neq\gamma$ and $\langle\beta,\gamma\rangle=0$.
\item
$\hat m_{\beta,\gamma}=3$ for $\beta,\gamma\in\Phi[\Gamma]$, $\beta\neq\gamma$ and $\langle\beta,\gamma\rangle=-1$.
\item
$\hat m_{\beta,\gamma}=\infty$ for $\beta,\gamma\in\Phi[\Gamma]$, $\beta\neq\gamma$ and $\langle\beta,\gamma\rangle\not\in\{0,-1\}$.
\end{itemize}
\end{expl2}

\begin{thm}\label{thm2_3}
Let $\Gamma$ be a Coxeter graph.
Then the map $\{\hat\delta_\beta\mid\beta\in\Phi[\Gamma]\}\to\{\delta_\beta\mid\beta\in\Phi[\Gamma]\}$, $\hat\delta_\beta\mapsto\delta_\beta$, induces an isomorphism $\varphi:A[\hat\Gamma]\to\KVA[\Gamma]$.
\end{thm}

\begin{rem}
When $\Gamma$ is the graph $A_{n-1}$ of Figure \ref{fig2_1} we get from Theorem \ref{thm2_3} the presentation of $\KVB_n=\KVA[\Gamma]$ given by Rabenda \cite{Raben1} and Bardakov--Bellingeri \cite{BarBel1}.
\end{rem}

\begin{proof}
We start by showing that $\{\delta_\beta\mid\beta\in\Phi[\Gamma]\}$ generates $\KVA[\Gamma]$.
Let $g\in\KVA[\Gamma]$.
There exist $w_0,w_1,\dots,w_p\in W[\Gamma]$, $s_1,\dots,s_p\in S$ and $\varepsilon_1,\dots,\varepsilon_p\in\{\pm 1\}$ such that:
\[
g=\iota_W(w_0)\,\sigma_{s_1}^{\varepsilon_1}\,\iota_W(w_1) \cdots \sigma_{s_p}^{\varepsilon_p}\,\iota_W(w_p)\,.
\]
For $i\in\{1,\dots,p\}$ we set $\beta_i=(w_0w_1\cdots w_{i-1})(\alpha_{s_i})$, and we set $w=w_0w_1\cdots w_p$.
Then $g=\delta_{\beta_1}^{\varepsilon_1} \cdots \delta_{\beta_p}^{\varepsilon_p}\, \iota_W(w)$.
Since $g\in\KVA[\Gamma]$, we have $1=\pi_K(g)=w$, hence $g=\delta_{\beta_1}^{\varepsilon_1}\cdots\delta_{\beta_p}^{\varepsilon_p}$.

Now we show that there is a homomorphism $\varphi:A[\hat\Gamma]\to\KVA[\Gamma]$ which sends $\hat\delta_\beta$ to $\delta_\beta$ for all $\beta\in\Phi[\Gamma]$.
Let $\beta,\gamma\in\Phi[\Gamma]$ such that $\beta\neq\gamma$ and $\hat m_{\beta,\gamma}\neq\infty$.
There exist $w\in W[\Gamma]$ and $s,t\in S$ such that $\beta=w(\alpha_s)$, $\gamma=w(\alpha_t)$ and $m_{s,t}=\hat m_{\beta,\gamma}$.
Then:
\begin{gather*}
\Prod_R(\delta_\gamma,\delta_\beta,\hat m_{\beta,\gamma})=
\iota_W(w)\,\Prod_R(\sigma_t,\sigma_s,m_{s,t})\,\iota_W(w)^{-1}=\\
\iota_W(w)\,\Prod_R(\sigma_s,\sigma_t,m_{s,t})\,\iota_W(w)^{-1}=
\Prod_R(\delta_\beta,\delta_\gamma,\hat m_{\beta,\gamma})\,.
\end{gather*}
Note that this homomorphism is surjective since we know that $\{\delta_\beta\mid\beta\in\Phi[\Gamma]\}$ generates $\KVA[\Gamma]$.

The action of $W[\Gamma]$ on $\{\hat\delta_\beta\mid\beta\in\Phi[\Gamma]\}$ defined by $w\cdot\hat\delta_\beta=\hat\delta_{w(\beta)}$ extends to an action of $W[\Gamma]$ on $A[\hat\Gamma]$.
We can therefore consider the semi-direct product $G=A[\hat\Gamma]\rtimes W[\Gamma]$ determined by this action.
Now, we show that the homomorphisms $\varphi:A[\hat\Gamma]\to\KVA[\Gamma]\subset\VA[\Gamma]$ and $\iota_W:W[\Gamma]\to\VA[\Gamma]$ induce a homomorphism $\tilde\varphi:G\to\VA[\Gamma]$.
In order to do this, we just need to check that $\varphi(w\cdot\hat\delta_\beta)=\iota_W(w)\,\varphi(\hat\delta_\beta)\,\iota_W(w)^{-1}$ for all $w\in W[\Gamma]$ and all $\beta\in\Phi[\Gamma]$, and this equality is easily verified since both terms are by definition equal to $\delta_{w(\beta)}$.

Let $\psi:\SS\sqcup\TT\to G$ be the map defined by $\psi(\sigma_s)=\hat\delta_{\alpha_s}$ and $\psi(\tau_s)=s$ for all $s\in S$.
Now we show that this map induces a homomorphism $\psi:\VA[\Gamma]\to G$.
Let $s,t\in S$, $s\neq t$, such that $m_{s,t}\neq\infty$.
Then:
\begin{gather*}
\Prod_R(\psi(\sigma_t),\psi(\sigma_s),m_{s,t})=
\Prod_R(\hat\delta_{\alpha_t},\hat\delta_{\alpha_s},\hat m_{\alpha_s,\alpha_t})=\\
\Prod_R(\hat\delta_{\alpha_s},\hat\delta_{\alpha_t},\hat m_{\alpha_s,\alpha_t})=
\Prod_R(\psi(\sigma_s),\psi(\sigma_t),m_{s,t})\,.
\end{gather*}
Let $s,t\in S$, $s\neq t$, such that $m_{s,t}\neq\infty$.
Then:
\[
\Prod_R(\psi(\tau_t),\psi(\tau_s),m_{s,t})=
\Prod_R(t,s,m_{s,t})=
\Prod_R(s,t,m_{s,t})=
\Prod_R(\psi(\tau_s),\psi(\tau_t),m_{s,t})\,.
\]
Let $s\in S$.
Then:
\[
\psi(\tau_s)^2=s^2=1\,.
\]
Let $s,t\in S$, $s\neq t$, such that $m_{s,t}\neq\infty$.
Set $r=s$ if $m_{s,t}$ is even and $r=t$ if $m_{s,t}$ is odd.
Let $w=\Prod_R(s,t,m_{s,t}-1)\in W[\Gamma]$.
Recall that, by Deodhar \cite{Deodh1}, $w(\alpha_s)=\alpha_r$.
Then: 
\begin{gather*}
\Prod_R(\psi(\tau_s),\psi(\tau_t),m_{s,t}-1)\,\psi(\sigma_s)=
\Prod_R(s,t,m_{s,t}-1)\,\hat\delta_{\alpha_s}=
w\hat\delta_{\alpha_s}=
\hat\delta_{w(\alpha_s)}w=\\
\hat\delta_{\alpha_r}\,\Prod_R(s,t,m_{s,t}-1)=
\psi(\sigma_r)\,\Prod_R(\psi(\tau_s),\psi(\tau_t),m_{s,t}-1)\,.
\end{gather*}

Now we prove that $\psi\circ\tilde\varphi=\id$.
This implies that $\tilde\varphi$ is injective, and therefore that $\varphi$ is injective, finishing the proof of the theorem.
Let $s\in S$.
Then $(\psi\circ\tilde\varphi)(s)=\psi(\tau_s)=s$.
Let $\beta\in\Phi[\Gamma]$.
Let $w\in W[\Gamma]$ and $s\in S$ such that $\beta=w(\alpha_s)$.
Then:
\begin{gather*}
(\psi\circ\tilde\varphi)(\hat\delta_\beta) =
\psi(\delta_\beta)=
\psi(\delta_{w(\alpha_s)})=
\psi(w\cdot\delta_{\alpha_s})=
\psi(\iota_W(w)\,\sigma_s\,\iota_W(w)^{-1})=
w\hat\delta_{\alpha_s}w^{-1}=\\
w\cdot\hat\delta_{\alpha_s}=
\hat\delta_{w(\alpha_s)}=
\hat\delta_\beta\,.
\end{gather*}
Since $S\cup\{\hat\delta_\beta\mid\beta\in\Phi[\Gamma]\}$ generates $G$, it follows that $\psi\circ\tilde\varphi=\id$.
\end{proof}

From now on we identify $\KVA[\Gamma]$ with $A[\hat\Gamma]$.
In particular, we assume that $\{\delta_\beta\mid\beta\in\Phi[\Gamma]\}$ is the standard generating set of $A[\hat\Gamma]$.

\begin{corl}\label{corl2_4}
Let $\Gamma$ be a Coxeter graph.
Then $\iota_A:A[\Gamma]\to\VA[\Gamma]$ is injective.
\end{corl}

\begin{proof}
Let $\Gamma$ be a Coxeter graph and let $M=(m_{s,t})_{s,t\in S}$ be its Coxeter matrix.
Recall that, for $X\subset S$, we set $M_X=(m_{s,t})_{s,t\in X}$ and we denote by $\Gamma_X$ the Coxeter graph of $M_X$.
Recall also that, by Van der Lek \cite{Lek1}, the homomorphism $\iota_X:A[\Gamma_X]\to A[\Gamma]$ induced by the inclusion is injective.

Let $\Pi=\{\alpha_s\mid s\in S\}$ be the set of simple roots.
Observe that the map $S\to\Pi$, $s\mapsto\alpha_s$, induces an isomorphism from $\Gamma$ to $\hat\Gamma_\Pi$, and therefore an isomorphism from $A[\Gamma]$ to $A[\hat\Gamma_\Pi]$.
By composing this isomorphism with the embedding $\iota_\Pi:A[\hat\Gamma_\Pi]\hookrightarrow A[\hat\Gamma]$, we get an injective homomorphism from $A[\Gamma]$ into $A[\hat\Gamma]=\KVA[\Gamma]$ which sends $s$ to $\delta_{\alpha_s}$ for all $s\in S$.
Since $\delta_{\alpha_s}=\sigma_s$ for all $s$, this homomorphism is nothing else than $\iota_A$.
\end{proof}

We turn now to determine a presentation for $\PVA[\Gamma]$.
Take an abstract set $\{\hat\zeta_\beta\mid\beta\in\Phi[\Gamma]\}$ in one-to-one correspondence with $\Phi[\Gamma]$.
Let $\beta,\gamma\in\Phi[\Gamma]$ such that $\hat m_{\beta,\gamma}\neq\infty$.
Let $m=\hat m_{\beta,\gamma}$.
We define roots $\beta_1,\dots,\beta_m\in\Phi[\Gamma]$ by $\beta_1=\beta$ and, for $k\ge2$,
\[
\beta_k=\left\{\begin{array}{ll}
\Prod_R(r_\gamma,r_\beta,k-1)(\gamma)&\text{if }k\text{ is even}\,,\\
\Prod_R(r_\beta,r_\gamma,k-1)(\beta)&\text{if }k\text{ is odd}\,.
\end{array}\right.
\]
We set 
\[
Z(\gamma,\beta,\hat m_{\beta,\gamma})=\hat\zeta_{\beta_m}\cdots\hat\zeta_{\beta_2}\hat\zeta_{\beta_1}
\]
that we consider as a word over $\{\hat\zeta_\beta\mid\beta\in\Phi[\Gamma]\}$.
We denote by $\widehat{\PVA}[\Gamma]$ the group defined by the presentation with generating set $\{\hat\zeta_\beta\mid\beta\in\Phi[\Gamma]\}$ and relations:
\[
Z(\gamma,\beta,\hat m_{\beta,\gamma})=Z(\beta,\gamma,\hat m_{\beta,\gamma}) \text{ for } \beta,\gamma\in\Phi[\Gamma], \ \beta\neq\gamma,  \ \hat m_{\beta,\gamma}\neq\infty.
\]

\begin{expl1}
Suppose that $\Gamma$ is the graph $A_{n-1}$ of Figure \ref{fig2_1}.
Recall that $W[\Gamma]$ is the symmetric group $S_n$, $A[\Gamma]$ is the braid group $B_n$, and $\VA[\Gamma]$ is the virtual braid group $\VB_n$.
Recall also that $\Pi=\{\alpha_{1,2},\dots,\alpha_{n-1,n}\}$ is the set of simple roots and $\Phi[\Gamma]=\{\alpha_{i,j}\mid1\le i\neq j\le n\}$ is the root system of $W[\Gamma]=S_n$.
Then $\widehat{\PVA}[A_{n-1}]$ is the group defined by the presentation with generating set $\{\hat\zeta_{\alpha_{i,j}}\mid1\le i\neq j\le n\}$ and relations:
\begin{itemize}
\item
$\hat\zeta_{\alpha_{i,j}}\hat\zeta_{\alpha_{k,\ell}}=\hat\zeta_{\alpha_{k,\ell}}\hat\zeta_{\alpha_{i,j}}$ for $i,j,k,\ell$ pairwise distinct,
\item
$\hat\zeta_{\alpha_{i,j}}\hat\zeta_{\alpha_{i,k}}\hat\zeta_{\alpha_{j,k}}=\hat\zeta_{\alpha_{j,k}}\hat\zeta_{\alpha_{i,k}}\hat\zeta_{\alpha_{i,j}}$ for $i,j,k$ pairwise distinct.
\end{itemize}
By Bardakov \cite{Barda1}, this group is isomorphic to $\PVB_n$.
\end{expl1}

\begin{expl2}
Suppose $\Gamma$ is simply laced.
Let $\beta,\gamma\in\Phi[\Gamma]$.
 If $\hat m_{\beta,\gamma}=2$, then $\langle\beta,\gamma\rangle=0$ and $r_\beta(\gamma)=\gamma$, hence $Z(\gamma,\beta,\hat m_{\beta,\gamma})=\hat\zeta_\gamma\hat\zeta_\beta$.
If $\hat m_{\beta,\gamma}=3$, then $\langle\beta,\gamma\rangle=-1$, $r_\beta(\gamma)=\beta+\gamma$ and $(r_\beta r_\gamma)(\beta)=\gamma$, hence $Z(\gamma,\beta,\hat m_{\beta,\gamma})=\hat\zeta_\gamma\hat\zeta_{\beta+\gamma}\hat\zeta_\beta$.
So, $\widehat{\PVA}[\Gamma]$ is the group defined by the presentation with generating set $\{\hat\zeta_\beta\mid\beta\in\Phi[\Gamma]\}$ and relations:
\begin{itemize}
\item
$\hat\zeta_\gamma\hat\zeta_\beta=\hat\zeta_\beta\hat\zeta_\gamma$ for $\beta,\gamma\in\Phi[\Gamma]$, $\beta\neq\gamma$ and $\langle\beta,\gamma\rangle=0$,
\item
$\hat\zeta_\gamma\hat\zeta_{\beta+\gamma}\hat\zeta_\beta=\hat\zeta_\beta\hat\zeta_{\beta+\gamma}\hat\zeta_\gamma$ for $\beta,\gamma\in\Phi[\Gamma]$, $\beta\neq\gamma$ and $\langle\beta,\gamma\rangle=-1$.
\end{itemize}
\end{expl2}

The following lemma will not be used in the paper, but it is useful for understanding the presentation of $\PVA[\Gamma]$ given in Theorem \ref{thm2_6}.

\begin{lem}\label{lem2_5}
Let $\Gamma$ be a Coxeter graph.
Let $\beta,\gamma\in\Phi[\Gamma]$ such that $\hat m_{\beta,\gamma}\neq\infty$.
Let $m=\hat m_{\beta,\gamma}$ and let $\beta_1,\dots,\beta_m\in\Phi[\Gamma]$ be the roots such that $Z(\gamma,\beta,\hat m_{\beta,\gamma})=\hat\zeta_{\beta_m}\cdots\hat\zeta_{\beta_2}\hat\zeta_{\beta_1}$.
Then $Z(\beta,\gamma,\hat m_{\beta,\gamma})$ is the reverse word of $Z(\gamma,\beta,\hat m_{\beta,\gamma})$, that is, $Z(\beta,\gamma,\hat m_{\beta,\gamma})=\hat\zeta_{\beta_1}\hat\zeta_{\beta_2}\cdots\hat\zeta_{\beta_m}$.
\end{lem}

\begin{proof}
We set $\beta_1=\beta$, $\gamma_1=\gamma$ and, for $k\ge2$,
\begin{gather*}
\beta_k=\left\{\begin{array}{ll}
\Prod_R(r_\gamma,r_\beta,k-1)(\gamma)&\text{if }k\text{ is even}\,,\\
\Prod_R(r_\beta,r_\gamma,k-1)(\beta)&\text{if }k\text{ is odd}\,,
\end{array}\right.\\
\gamma_k=\left\{\begin{array}{ll}
\Prod_R(r_\beta,r_\gamma,k-1)(\beta)&\text{if }k\text{ is even}\,,\\
\Prod_R(r_\gamma,r_\beta,k-1)(\gamma)&\text{if }k\text{ is odd}\,.
\end{array}\right.
\end{gather*}
We have: 
\[
Z(\gamma,\beta,\hat m_{\beta,\gamma})=\hat\zeta_{\beta_m}\cdots\hat\zeta_{\beta_2}\hat\zeta_{\beta_1} \text{ and }
Z(\beta,\gamma,\hat m_{\beta,\gamma})=\hat\zeta_{\gamma_m}\cdots\hat\zeta_{\gamma_2} \hat\zeta_{\gamma_1}\,.
\]
We need to show that $\beta_k=\gamma_{m+1-k}$ for all $k\in\{1,\dots,m\}$.
We assume that $m$ is even.
The case where $m$ is odd can be proved in the same way.
Recall that, by Deodhar \cite{Deodh1}, we have:
\[
\gamma=\Prod_R(r_\gamma,r_\beta,m-1)(\gamma)\text{ and }\beta=\Prod_R(r_\beta,r_\gamma,m-1)(\beta)\,.
\]
If $k$ is even, then $m+1-k$ is odd, hence:
\begin{gather*}
\gamma_{m+1-k}=
\Prod_R(r_\gamma,r_\beta,m-k)(\gamma)=
\Prod_R(r_\gamma,r_\beta,m-k)\,\Prod_R(r_\gamma,r_\beta,m-1)(\gamma)=\\
\Prod_R(r_\gamma,r_\beta,m-k)\,\Prod_L(r_\beta,r_\gamma,m-1)(\gamma)=
\Prod_L(r_\beta,r_\gamma,k-1)(\gamma)=\\
\Prod_R(r_\gamma,r_\beta,k-1)(\gamma)=
\beta_k\,.
\end{gather*}
If $k$ is odd, then $m+1-k$ is even, hence:
\begin{gather*}
\gamma_{m+1-k}=
\Prod_R(r_\beta,r_\gamma,m-k)(\beta)=
\Prod_R(r_\beta,r_\gamma,m-k)\,\Prod_R(r_\beta,r_\gamma,m-1)(\beta)=\\
\Prod_R(r_\beta,r_\gamma,m-k)\,\Prod_L(r_\gamma,r_\beta,m-1)(\beta)=
\Prod_L(r_\beta,r_\gamma,k-1)(\beta)=\\
\Prod_R(r_\beta,r_\gamma,k-1)(\beta)=
\beta_k\,.\qquad\proved
\end{gather*}
\end{proof}

\begin{thm}\label{thm2_6}
Let $\Gamma$ be a Coxeter graph.
Then the map $\{\hat\zeta_\beta\mid\beta\in\Phi[\Gamma]\}\to\{\zeta_\beta\mid\beta\in\Phi[\Gamma]\}$, $\hat\zeta_\beta\mapsto\zeta_\beta$, induces an isomorphism $\varphi:\widehat{\PVA}[\Gamma]\to\PVA[\Gamma]$.
\end{thm}

\begin{proof}
The proof is similar to that of Theorem \ref{thm2_3}.
One can easily show that $\{\zeta_\beta\mid\beta\in\Phi[\Gamma]\}$ generates $\PVA[\Gamma]$ in the same way as we shown that $\{\delta_\beta\mid\beta\in\Phi[\Gamma]\}$ generates $\KVA[\Gamma]$ in the proof of Theorem \ref{thm2_3}.

Now, we show that we have a homomorphism $\varphi:\widehat{\PVA}[\Gamma]\to\PVA[\Gamma]$ which sends $\hat\zeta_\beta$ to $\zeta_\beta$ for all $\beta\in\Phi[\Gamma]$.
Let $\beta,\gamma\in\Phi[\Gamma]$ such that $\hat m_{\beta,\gamma}\neq\infty$.
There exist $w\in W$ and $s,t\in S$ such that $\beta=w(\alpha_s)$, $\gamma=w(\alpha_t)$, and $m_{s,t}=\hat m_{\beta,\gamma}$.
Recall the definition of $Z(\gamma,\beta,\hat m_{\beta,\gamma})$.
Let $m =\hat m_{\beta,\gamma}=m_{s,t}$ and let $\beta_1,\dots,\beta_m\in\Phi[\Gamma]$ be the roots defined by $\beta_1=\beta$ and, for $k\ge 2$,
\[
\beta_k=\left\{\begin{array}{ll}
\Prod_R(r_\gamma,r_\beta,k-1)(\gamma)&\text{if }k\text{ is even}\,,\\
\Prod_R(r_\beta,r_\gamma,k-1)(\beta)&\text{if }k\text{ is odd}\,.
\end{array}\right.
\]
Then $Z(\gamma,\beta,\hat m_{\beta,\gamma})=\hat\zeta_{\beta_m}\cdots\hat\zeta_{\beta_2}\hat\zeta_{\beta_1}$.
Let $k\in\{1,\dots,m\}$.
Suppose that $k$ is odd.
Then:
\[
\beta_k=
\Prod_R(r_\beta,r_\gamma,k-1)(\beta)=
(w\,\Prod_R(s,t,k-1)\,w^{-1} w)(\alpha_s)=
(w\,\Prod_R(s,t,k-1))(\alpha_s)\,,
\]
hence:
\begin{gather*}
\zeta_{\beta_k}=
w\,\Prod_R(s,t,k-1) \cdot \zeta_{\alpha_s}=\\
\iota_W(w)\,\Prod_R(\tau_s,\tau_t,k-1)\,(\tau_s\sigma_s)\,\Prod_R(\tau_s,\tau_t,k-1)^{-1} \iota_W(w)^{-1} =\\
\iota_W(w)\,\Prod_R(\tau_t,\tau_s,k)\, \sigma_s\,\Prod_R(\tau_s,\tau_t,k-1)^{-1} \iota_W(w)^{-1}\,.
\end{gather*}
We show in the same way that, if $k$ is even, then: 
\[
\zeta_{\beta_k}=
\iota_W(w)\,\Prod_R(\tau_s,\tau_t,k)\, \sigma_t\, \Prod_R(\tau_t,\tau_s,k-1)^{-1}\iota_W(w)^{-1}\,.
\]
Thus: 
\[
\zeta_{\beta_m}\cdots\zeta_{\beta_2}\zeta_{\beta_1}=
\iota_W(w)\,\Prod_L(\tau_s,\tau_t,m)\,\Prod_R(\sigma_t,\sigma_s,m)\,\iota_W(w)^{-1}\,.
\]
Similarly, if $\gamma_1,\dots,\gamma_m$ are the roots such that $Z(\beta,\gamma,\hat m_{\beta,\gamma})=\hat\zeta_{\gamma_m}\cdots\hat\zeta_{\gamma_2}\hat\zeta_{\gamma_1}$, then:
\[
\zeta_{\gamma_m}\cdots\zeta_{\gamma_2}\zeta_{\gamma_1}=\iota_W(w)\,\Prod_L(\tau_t,\tau_s,m)\,\Prod_R(\sigma_s,\sigma_t,m)\,\iota_W(w)^{-1}\,.
\]
So, $\zeta_{\beta_m}\cdots\zeta_{\beta_2}\zeta_{\beta_1}=\zeta_{\gamma_m}\cdots\zeta_{\gamma_2}\zeta_{\gamma_1}$.

The action of $W[\Gamma]$ on $\{\hat\zeta_\beta\mid\beta\in\Phi[\Gamma]\}$ defined by $w\cdot\hat\zeta_\beta=\hat\zeta_{w(\beta)}$ extends to an action of $W[\Gamma]$ on $\widehat{\PVA}[\Gamma]$.
Thus, we can consider the semi-direct product $G=\widehat{\PVA}[\Gamma]\rtimes W[\Gamma]$ determined by this action.
Now, we show that the homomorphisms $\varphi:\widehat{\PVA}[\Gamma]\to\PVA[\Gamma]\subset\VA[\Gamma]$ and $\iota_W:W[\Gamma]\to\VA[\Gamma]$ induce a homomorphism $\tilde\varphi:G\to\VA[\Gamma]$.
To do this, we just need to check that $\varphi(w\cdot\hat\zeta_\beta)=\iota_W(w)\,\varphi(\hat\zeta_\beta)\,\iota_W(w)^{-1}$ for all $w\in W[\Gamma]$ and all $\beta\in\Phi[\Gamma]$, and this equality is easily verified since both terms are by definition equal to $\zeta_{w(\beta)}$.

Let $\psi:\SS\sqcup\TT\to G$ be the map defined by $\psi(\sigma_s)=s\hat\zeta_{\alpha_s}$ and $\psi(\tau_s)=s$ for all $s\in S$.
Now, we show that this map induces a homomorphism $\psi:\VA[\Gamma]\to G$.
Let $s,t\in S$, $s\neq t$, such that $m_{s,t}\neq\infty$.
Then:
\[
\Prod_R(\psi(\tau_t),\psi(\tau_s),m_{s,t})=
\Prod_R(t,s,m_{s,t})=
\Prod_R(s,t,m_{s,t})=
\Prod_R(\psi(\tau_s),\psi(\tau_t),m_{s,t})\,.
\]
Let $s\in S$.
Then:
\[
\psi(\tau_s)^2=s^2=1\,.
\]
Let $s,t\in S$, $s\neq t$, such that $m_{s,t}\neq\infty$.
Let $m=m_{s,t}$ and let $\beta_1,\dots,\beta_m\in\Phi[\Gamma]$ be the roots defined by $\beta_1=\alpha_s$ and, for $k\ge 2$,
\[
\beta_k=\left\{\begin{array}{ll}
\Prod_R(t,s,k-1)(\alpha_t)&\text{if }k\text{ is even}\,,\\
\Prod_R(s,t,k-1)(\alpha_s)&\text{if }k\text{ is odd}\,.
\end{array}\right.
\]
Then:
\begin{gather*}
\Prod_R(\psi(\sigma_t),\psi(\sigma_s),m_{s,t})=
\cdots s\hat\zeta_{\alpha_s} t\hat\zeta_{\alpha_t} s\hat\zeta_{\alpha_s}=
\Prod_R(t,s,m_{s,t})\,\hat\zeta_{\beta_m} \cdots \hat\zeta_{\beta_2} \hat\zeta_{\beta_1}=\\
\Prod_R(t,s,m_{s,t})\,Z(\alpha_t,\alpha_s,\hat m_{\alpha_s,\alpha_t})\,.
\end{gather*}
Similarly, 
\[
\Prod_R(\psi(\sigma_s),\psi(\sigma_t),m_{s,t})=
\Prod_R(s,t,m_{s,t})\,Z(\alpha_s,\alpha_t,\hat m_{\alpha_s,\alpha_t})\,.
\]
Hence:
\[
\Prod_R(\psi(\sigma_t),\psi(\sigma_s),m_{s,t})=
\Prod_R(\psi(\sigma_s),\psi(\sigma_t),m_{s,t})\,.
\]
Let $s,t\in S$, $s\neq t$, such that $m_{s,t}\neq\infty$.
Set $r=s$ if $m_{s,t}$ is even and $r=t$ if $m_{s,t}$ is odd.
Then the following equality can be proved in the same way as the corresponding equality in the proof of Theorem~\ref{thm2_3}:
\[
\Prod_R(\psi(\tau_s),\psi(\tau_t),m_{s,t}-1)\,\psi(\sigma_s) =
\psi(\sigma_r)\,\Prod_R(\psi(\tau_s),\psi(\tau_t),m_{s,t}-1)\,.
\]

Finally, we prove that $\psi\circ\tilde\varphi=\id$ in the same way as in the proof of Theorem \ref{thm2_3}, which implies the injectivity of $\tilde\varphi$, hence of $\varphi$, and concludes the proof of Theorem \ref{thm2_6}.
\end{proof}


\section{Center of $\VA[\Gamma]$}\label{sec3}

We start this section by recalling some results on centers of Coxeter groups.
If $\Gamma$ is a Coxeter graph and $M=(m_{s,t})_{s,t\in S}$ is its Coxeter matrix, then we denote by $\lg:W[\Gamma]\to\N$ the word length in $W[\Gamma]$ with respect to $S$.
The following can be found in Bourbaki \cite[Exercice 22, Page 22]{Bourb1}.

\begin{prop}[Bourbaki \cite{Bourb1}]\label{prop3_1}
Let $\Gamma$ be a Coxeter graph and let $M=(m_{s,t})_{s,t\in S}$ be its Coxeter matrix.
The following conditions on an element $w_0\in W[\Gamma]$ are equivalent:
\begin{itemize}
\item[(a)]
For all $u\in W[\Gamma]$, $\lg(w_0)=\lg(u)+\lg(u^{-1}w_0)$.
\item[(b)]
For all $s\in S$, $\lg(w_0)>\lg(sw_0)$.
\end{itemize}
Moreover, $w_0$ exists if and only if $W[\Gamma]$ is finite.
If $w_0$ satisfies (a) and/or (b), then $w_0$ is unique, $w_0$ is an involution, and $w_0Sw_0=S$.
\end{prop}

The element $w_0$ of Proposition \ref{prop3_1}, if it exists, is called \emph{the longest element} of $W[\Gamma]$.
The center of a group $G$ will be denoted by $Z(G)$.
The following can be found in Bourbaki \cite[Exercice 3, Page 130]{Bourb1}.

\begin{prop}[Bourbaki \cite{Bourb1}]\label{prop3_2}
Let $\Gamma$ be a connected Coxeter graph.
\begin{itemize}
\item[(1)]
If $W[\Gamma]$ is infinite, then $Z(W[\Gamma])=\{1\}$.
\item[(2)]
Suppose that $W[\Gamma]$ is finite.
Let $w_0$ be the longest element of $W[\Gamma]$.
Then $Z(W[\Gamma])=\{1,w_0\}$ if $w_0$ is central, and $Z(W[\Gamma])=\{1\}$ otherwise.
\end{itemize}
\end{prop}

Now, we return to the study of virtual Artin groups.

\begin{thm}\label{thm3_3}
Let $\Gamma$ be a Coxeter graph.
Assume that $W[\Gamma]$ is a subgroup of $\VA[\Gamma]$ via the embedding $\iota_W:W[\Gamma]\hookrightarrow \VA[\Gamma]$.
Then the centralizer of $W[\Gamma]$ in $\VA[\Gamma]$ is $Z(W[\Gamma])$.
\end{thm}

\begin{proof}
If $\XX$ is a subset of $\Phi[\Gamma]$, then we denote by $\hat\Gamma_\XX$ the full subgraph of $\hat\Gamma$ spanned by $\XX$.
By Van der Lek \cite[Lemma 4.11]{Lek1} we can identify $A[\hat\Gamma_\XX]$ with the subgroup of $A[\hat\Gamma]=\KVA[\Gamma]$ generated by $\{\delta_\beta\mid\beta\in\XX\}$.
Moreover, again by Van der Lek \cite[Lemma 4.12]{Lek1}, if $\{\XX_i\mid i\in I\}$ is a family of subsets of $\Phi[\Gamma]$, then:
\begin{equation}\label{eq3_1}
\bigcap_{i\in I}A[\hat\Gamma_{\XX_i}]=A[\hat\Gamma_{\cap_{i\in I} \XX_i}]\,.
\end{equation}

It is clear that $Z(W[\Gamma])$ is contained in the centralizer of $W[\Gamma]$ in $\VA[\Gamma]$.
So, we just have to prove that the centralizer of $W[\Gamma]$ in $\VA[\Gamma]$ is contained in $Z(W[\Gamma])$.
Let $g$ be an element of the centralizer of $W[\Gamma]$ in $\VA[\Gamma]$.
We write $g$ in the form $g=hw$ with $h\in\KVA[\Gamma]$ and $w\in W[\Gamma]$.
The element $\pi_K(g)=w$ must lie in the center of $W[\Gamma]$, hence $h$ also lies in the centralizer of $W[\Gamma]$ in $\VA[\Gamma]$.
Let $\beta\in\Phi^+[\Gamma]$.
We set $\XX^+_\beta=\Phi[\Gamma]\setminus\{-\beta\}$, $\XX^-_\beta=\Phi[\Gamma]\setminus\{\beta\}$ and $\YY_\beta=\Phi[\Gamma]\setminus\{\beta,-\beta\}$.
Since $\hat m_{\beta,-\beta}=\infty$, $\XX^+_\beta\cup\XX^-_\beta=\Phi[\Gamma]$ and $\XX^+_\beta\cap\XX^-_\beta=\YY_\beta$, by Bellingeri--Paris \cite[Lemma 3.3]{BelPar1},
\[
\KVA[\Gamma]=A[\hat\Gamma]=A[\hat\Gamma_{\XX^+_\beta}]*_{A[\hat\Gamma_{\YY_\beta}]} A[\hat\Gamma_{\XX^-_\beta}]\,.
\]
Since $r_\beta(\beta)=-\beta$ and $r_\beta(-\beta)=\beta$, we have $r_\beta(\XX^+_\beta)=\XX^-_\beta$, $r_\beta(\XX^-_\beta)=\XX^+_\beta$ and $r_\beta(\YY_\beta)=\YY_\beta$, hence $r_\beta\cdot A[\hat\Gamma_{\XX^+_\beta}]=A[\hat\Gamma_{\XX^-_\beta}]$, $r_\beta \cdot A[\hat\Gamma_{\XX^-_\beta}]=A[\hat\Gamma_{\XX^+_\beta}]$ and $r_\beta\cdot A[\hat\Gamma_{\YY_\beta}]=A[\hat\Gamma_{\YY_\beta}]$.
Since $r_\beta\in W[\Gamma]$ and $h$ lies in the centralizer of $W[\Gamma]$ in $\VA[\Gamma]$, $r_\beta$ and $h$ commute, which means in terms of actions that $r_\beta \cdot h=h$.
By Bellingeri--Paris \cite[Lemma 3.6]{BelPar1}, it follows that $h\in A[\hat\Gamma_{\YY_\beta}]$.
We have $\cap_{\beta\in\Phi^+[\Gamma]}\YY_\beta=\emptyset$, hence, by Equation (\ref{eq3_1}),
\[
h\in\bigcap_{\beta\in\Phi^+[\Gamma]}A[\hat\Gamma_{\YY_\beta}]=A[\hat\Gamma_\emptyset]=\{1\}\,.
\]
Thus, $h=1$ and $g=w\in Z(W[\Gamma])$.
\end{proof}

\begin{corl}\label{corl3_4}
Let $\Gamma$ be a Coxeter graph.
Then $Z(\VA[\Gamma])=\{1\}$.
\end{corl}

\begin{proof}
Let $\Gamma_1,\dots,\Gamma_\ell$ be the connected components of $\Gamma$.
We have $\VA[\Gamma]=\VA[\Gamma_1]\times\cdots\times\VA[\Gamma_\ell]$, hence:
\[
Z(\VA[\Gamma])=Z(\VA[\Gamma_1])\times\cdots\times Z(\VA[\Gamma_\ell])\,.
\]
So, we can assume that $\Gamma$ is connected.
Note that the center of $\VA[\Gamma]$ is contained in the centralizer of $W[\Gamma]$ in $\VA[\Gamma]$.
If either $W[\Gamma]$ is infinite or $W[\Gamma]$ is finite and the longest element $w_0$ is not central, then $Z(\VA[\Gamma])=\{1\}$ by an easy combination of Proposition \ref{prop3_2} and Theorem \ref{thm3_3}.
Suppose $W[\Gamma]$ is finite and $w_0$ is central.
Then, by Proposition \ref{prop3_2} and Theorem \ref{thm3_3}, $Z(\VA[\Gamma])\subset\{1,w_0\}$.
Let $s\in S$.
By Deodhar \cite{Deodh1} we have $w_0(\alpha_s)=-\alpha_s$, hence $w_0\cdot\delta_{\alpha_s} = \delta_{-\alpha_s} \neq \delta_{\alpha_s}$, thus $w_0$ and $\delta_{\alpha_s}=\sigma_s$ do not commute.
So, $w_0\not\in Z(\VA[\Gamma])$, hence $Z(\VA[\Gamma])=\{1\}$.
\end{proof}

\begin{rem}
Determining the center of $A[\Gamma]$ for any Coxeter graph is an open problem.
\end{rem}

\begin{rem}
The fact that  $Z(\VB_n)=\{1\}$  has been proved in \cite{DieNic1} using  the construction by Bartholdi--Enriquez--Etingof--Rains \cite{BaEnEtRa1} of a classifying space for $\PVB_n$.
\end{rem}

We can use the same techniques to prove the following which is also a direct consequence of Charney--Morris-Wright \cite{ChaMor1}.

\begin{prop}\label{prop3_5}
Let $\Gamma$ be a Coxeter graph.
Then $Z(\KVA[\Gamma])=\{1\}$.
\end{prop}

\begin{proof}
Let $\beta\in\Phi^+[\Gamma]$.
We set $\XX^+_\beta=\Phi[\Gamma]\setminus\{-\beta\}$, $\XX^-_\beta=\Phi[\Gamma]\setminus\{\beta\}$ and $\YY_\beta=\Phi[\Gamma]\setminus\{\beta,-\beta\}$.
As in the proof of Theorem \ref{thm3_3}, by Bellingeri--Paris \cite[Lemma 3.3]{BelPar1}, from the equalities $\hat m_{\beta,-\beta}=\infty$, $\XX^+_\beta\cup\XX^-_\beta=\Phi[\Gamma]$ and $\XX^+_\beta\cap\XX^-_\beta=\YY_\beta$ it follows that:
\[
\KVA[\Gamma]=A[\hat\Gamma]=A[\hat\Gamma_{\XX^+_\beta}]*_{A[\hat\Gamma_{\YY_\beta}]} A[\hat\Gamma_{\XX^-_\beta}]\,.
\]
It is well-known that the center of an amalgamated product is a subgroup of the amalgam, hence $Z(\KVA[\Gamma])\subset Z(A[\hat\Gamma_{\YY_\beta}])$.
Since we know from the proof of Theorem \ref{thm3_3} that $\cap_{\beta\in\Phi^+[\Gamma]}A[\hat\Gamma_{\YY_\beta}]=\{1\}$, we conclude that $Z(\KVA[\Gamma])=\{1\}$.
\end{proof}


\section{Virtual Artin groups of spherical type and of affine type}\label{sec4}

Let $\Gamma$ be a finite Coxeter graph.
We say that $\Gamma$ (or $A[\Gamma]$ or $W[\Gamma]$ or $\VA[\Gamma]$) is of \emph{spherical type} if $W[\Gamma]$ is finite.
If $\Gamma_1,\dots,\Gamma_p$ are the connected components of $\Gamma$, then $W[\Gamma] = W[\Gamma_1] \times\cdots\times W[\Gamma_p]$, hence $\Gamma$ is of spherical type if and only if all its connected components are of spherical type.
The connected Coxeter graphs of spherical type have been classified by Coxeter himself \cite{Coxet1}.
These are the Coxeter graphs $A_n$ ($n\ge 1$), $B_n$ ($n\ge 2$), $D_n$ ($n\ge 4$), $E_n$ ($n\in\{6,7,8\}$), $F_4$, $H_3$, $H_4$ and $I_2(m)$ $(m\ge 5$) shown in Figure \ref{fig4_1}.

\begin{figure}[ht!]
\begin{center}
$A_n$:\quad \parbox[c]{4cm}{\includegraphics[width=4cm]{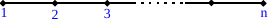}}\hskip 1 cm
$B_n$:\quad \parbox[c]{4cm}{\includegraphics[width=4cm]{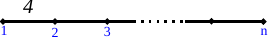}}

\bigskip
$D_n$:\quad \parbox[c]{4cm}{\includegraphics[width=4cm]{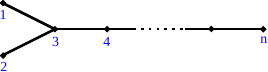}}\hskip 1 cm
$E_6$:\quad \parbox[c]{3.2cm}{\includegraphics[width=3.2cm]{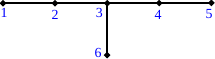}}

\bigskip
$E_7$:\quad \parbox[c]{4cm}{\includegraphics[width=4cm]{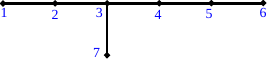}}\hskip 1 cm
$E_8$:\quad \parbox[c]{4.8cm}{\includegraphics[width=4.8cm]{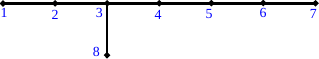}}

\bigskip
$F_4$:\quad \parbox[c]{2.4cm}{\includegraphics[width=2.4cm]{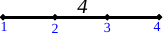}}\hskip 1 cm
$H_3$:\quad \parbox[c]{1.6cm}{\includegraphics[width=1.6cm]{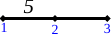}}\hskip 1 cm
$H_4$:\quad \parbox[c]{2.4cm}{\includegraphics[width=2.4cm]{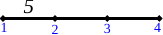}}

\bigskip
$I_2(m)$:\quad \parbox[c]{0.8cm}{\includegraphics[width=0.8cm]{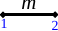}}

\caption{Connected Coxeter graphs of spherical type}\label{fig4_1}
\end{center}
\end{figure}

Let $\Gamma$ be a finite Coxeter graph and let $M=(m_{s,t})_{s,t\in S}$ be its Coxeter matrix.
Let $\Pi=\{\alpha_s\mid s\in S\}$ be the set of simple roots, let $V$ be the real vector space having $\Pi$ as a basis and let $\langle\cdot,\cdot\rangle:V\times V\to\R$ be the canonical bilinear form of $W[\Gamma]$.
We know by Coxeter \cite{Coxet2,Coxet1} that $W[\Gamma]$ is finite if and only if $\langle\cdot,\cdot\rangle$ is positive definite.
We say that $\Gamma$ (or $A[\Gamma]$ or $W[\Gamma]$ or $\VA[\Gamma]$) is of \emph{affine type} if $\langle\cdot,\cdot\rangle$ is positive but not positive definite.
It is easily seen that $\Gamma$ is of affine type if and only if all its connected components are of affine type or of spherical type and at least one of the connected components is of affine type.
On the other hand a classification of the connected Coxeter graphs of affine type can be found in Bourbaki \cite{Bourb1}.
These are the Coxeter graphs $\tilde A_n$ ($n\ge 1$), $\tilde B_n$ ($n\ge 3$), $\tilde C_n$ ($n\ge 2$), $\tilde D_n$ ($n\ge 4$), $\tilde E_n$ ($n\in\{6,7,8\}$), $\tilde F_4$ and $\tilde G_2$ shown in Figure \ref{fig4_2}.

\begin{figure}[ht!]
\begin{center}
$\tilde A_1$:\quad \parbox[c]{0.9cm}{\includegraphics[width=0.9cm]{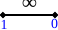}}\hskip 0.5 cm
$\tilde A_n$:\quad \parbox[c]{4cm}{\includegraphics[width=4cm]{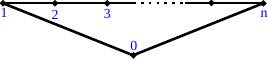}}\hskip 1 cm
$\tilde B_n$:\quad \parbox[c]{4.8cm}{\includegraphics[width=4.8cm]{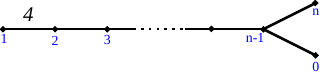}}

\bigskip
$\tilde C_n$:\quad \parbox[c]{4.8cm}{\includegraphics[width=4.8cm]{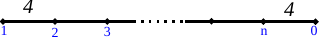}}\hskip 1 cm
$\tilde D_n$:\quad \parbox[c]{4.8cm}{\includegraphics[width=4.8cm]{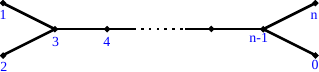}}

\bigskip
$\tilde E_6$:\quad \parbox[c]{3.2cm}{\includegraphics[width=3.2cm]{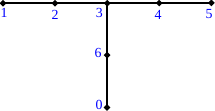}}\hskip 1 cm
$\tilde E_7$:\quad \parbox[c]{4.8cm}{\includegraphics[width=4.8cm]{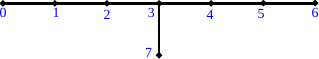}}

\bigskip
$\tilde E_8$:\quad \parbox[c]{5.6cm}{\includegraphics[width=5.6cm]{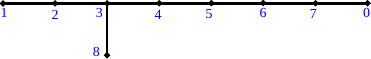}}\hskip 1 cm
$\tilde F_4$:\quad \parbox[c]{3.2cm}{\includegraphics[width=3.2cm]{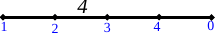}}

\bigskip
$\tilde G_2$:\quad \parbox[c]{1.6cm}{\includegraphics[width=1.6cm]{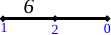}}

\caption{Connected Coxeter graphs of affine type}\label{fig4_2}
\end{center}
\end{figure}

Artin groups of spherical type are well understood since Artin groups themselves were introduced by Brieskorn \cite{Bries1}, Deligne \cite{Delig1} and Brieskorn--Saito \cite{BriSai1} in the early 1970s.
One would have expected that the next class of Artin groups to be well understood would be the class  of Artin groups of affine type, but this was not the case.
It was not until the last decade, with the development of the dual structures and the work of Digne \cite{Digne1,Digne2}, McCammond \cite{McCam1}, McCammond--Sulway \cite{McCSul1} and Paolini--Salvetti \cite{PaoSal1}, that elementary problems on these groups such as determining the center, determining the torsion part, or solving the word problem, have been solved, as well as the famous $K(\pi,1)$ conjecture (see Section \ref{sec6}).
In contrast, virtual Artin groups of affine type are treated more or less in the same way as those of spherical type.
Their most important common point is the following Theorem \ref{thm4_1} which allows us to understand the free of infinity subsets of $\Phi[\Gamma]$, viewed as the set of vertices of $\hat\Gamma$.

Let $\Gamma$ be a Coxeter graph and let $M=(m_{s,t})_{s,t\in S}$ be its Coxeter matrix.
We denote by $\PP_\sph(S)$ the set of subsets $X\subset S$ such that $\Gamma_X$ is of spherical type.
Then the \emph{spherical dimension} of $A[\Gamma]$ is defined to be $n_\sph(A[\Gamma])=\max\{|X|\mid X\in\PP_\sph(S)\}$.
We know by Charney--Davis \cite{ChaDav1} that, if $A[\Gamma]$ satisfies the $K(\pi,1)$ conjecture, then $n_\sph(A[\Gamma])$ is equal to the cohomological dimension of $A[\Gamma]$, and it is conjectured that all Artin groups satisfy the $K(\pi,1)$ conjecture.
This conjecture will be recalled in the beginning of Section \ref{sec6}.
Note however that, by Deligne \cite{Delig1} and Paolini--Salvetti \cite{PaoSal1}, $A[\Gamma]$ satisfies the $K(\pi,1)$ conjecture if $\Gamma$ is of spherical type or of affine type.

\begin{expl}
Suppose that $\Gamma$ is of spherical type or of affine type.
Let  $\langle\cdot,\cdot\rangle:V\times V\to\R$ be the canonical bilinear form of $W[\Gamma]$.
Then $n_\sph(A[\Gamma])$ is equal to the rank of $\langle\cdot,\cdot\rangle$.
In particular, if $\Gamma$ is of spherical type, then $n_\sph(A[\Gamma])$ is equal to the number of vertices of $\Gamma$, and if $\Gamma$ is connected of affine type, then $n_\sph(A[\Gamma])$ is equal to the number of vertices of $\Gamma$ minus $1$.
\end{expl}

Let $\Gamma$ be a Coxeter graph and let $M=(m_{s,t})_{s,t\in S}$ be its Coxeter matrix.
A subset $X\subset S$ is called \emph{free of infinity} if $m_{s,t}\neq\infty$ for all $s,t\in X$.
We denote by $\PP_{<\infty}(S)$ the set of free of infinity subsets of $S$.
Note that $\PP_\sph(S)\subset\PP_{<\infty}(S)$, hence:
\[
n_\sph(A[\Gamma])=\max\{n_\sph(A[\Gamma_X])\mid X\in\PP_{<\infty}(S)\}\,.
\]

The following theorem is the main result of the present section.
Recall the general principle which says that, if we understand $A[\Gamma_X]$ for all $X\in\PP_{<\infty}(S)$, then we can understand $A[\Gamma]$ (see Godelle--Paris \cite{GodPar1}).
Since we understand the Artin groups of spherical type and the Artin groups of affine type, this theorem enables to apply this principle and therefore to understand $A[\hat\Gamma]=\KVA[\Gamma]$ when $\Gamma$ is of spherical type or of affine type.
Since $\VA[\Gamma]=\KVA[\Gamma]\rtimes W[\Gamma]$ and Coxeter groups are well-understood, we can therefore deduce results on $\VA[\Gamma]$ itself.
We will apply this strategy in Section \ref{sec5} and Section \ref{sec6} to prove that $\VA[\Gamma]$ has a solution to the word problem and to compute the virtual cohomological dimension of $\VA[\Gamma]$ when $\Gamma$ is of spherical type or of affine type.

\begin{thm}\label{thm4_1}
Let $\Gamma $ be a Coxeter graph of spherical type or of affine type, and let $\XX$ be a free of infinity subset of $\Phi[\Gamma]$.
Then $\XX$ is finite, $\hat\Gamma_\XX$ is of spherical type or of affine type, and $n_\sph(A[\hat\Gamma_\XX])\le n_\sph(A[\Gamma])$.
\end{thm}

\begin{proof}
A symmetric matrix $A=(a_{i,j})_{1\le i,j\le n}$ is called \emph{decomposable} if there exists a non-trivial partition of the set of indices, $\{1,\dots,n\}=I\sqcup J$, such that $a_{i,j}=a_{j,i}=0$ for all $i\in I$ and $j\in J$.
We say that $A$ is \emph{indecomposable} otherwise.
The following claim is well-known and can be found for example in Bourbaki \cite[Lemme 4, Page 78]{Bourb1}.

{\it Claim 1.}
Let $A=(a_{i,j})_{1\le i,j\le n}$ be a positive symmetric matrix but not positive definite.
Suppose that $A$ is indecomposable, that $a_{i,i}>0$ for all $i\in\{1,\dots,n\}$, and that $a_{i,j}\le 0$ for all $i,j\in\{1,\dots,n\}$, $i\neq j$.
\begin{itemize}
\item[(1)]
The rank of $A$ is $n-1$.
\item[(2)]
If $X$ is a proper subset of $\{1,\dots,n\}$, then the matrix $A_X=(a_{i,j})_{i,j\in X}$ is positive definite.
\end{itemize}

Now take a Coxeter graph $\Gamma$ of spherical type or of affine type, and denote by $M=(m_{s,t})_{s,t\in S}$ its Coxeter matrix.
Let $\hat\Pi=\{\hat E_\beta\mid\beta\in\Phi[\Gamma]\}$ be a set in one-to-one correspondence with $\Phi[\Gamma]$, and let $\hat V=\oplus_{\beta\in\Phi[\Gamma]}\R\,\hat E_\beta$ be a real vector space having $\hat\Pi$ as a basis. 
Define the symmetric bilinear form $\langle\!\langle\cdot,\cdot\rangle\!\rangle:\hat V\times\hat V\to\R$ by $\langle\!\langle \hat E_\beta,\hat E_\gamma\rangle\!\rangle=-2\cos(\pi/\hat m_{\beta,\gamma})$ if $\hat m_{\beta,\gamma}\neq\infty$ and $\langle\!\langle \hat E_\beta,\hat E_\gamma\rangle\!\rangle=-2$ if $\hat m_{\beta,\gamma}=\infty$.

{\it Claim 2.}
Let $\beta,\gamma\in\Phi[\Gamma]$ such that $\hat m_{\beta,\gamma}\neq\infty$.
Then $\langle\!\langle\hat E_\beta,\hat E_\gamma\rangle\!\rangle=\langle\beta,\gamma\rangle$.

{\it Proof of Claim 2.}
If $\beta=\gamma$, then $\langle\!\langle\hat E_\beta,\hat E_\gamma\rangle\!\rangle=\langle\beta,\gamma\rangle=2$.
Suppose $\beta\neq \gamma$.
By definition there exist $w\in W$ and $s,t\in S$ such that $m_{s,t}\neq\infty$, $\beta=w(\alpha_s)$, $\gamma=w(\alpha_t)$ and $\hat m_{\beta,\gamma}=m_{s,t}$.
Then:
\[
\langle\beta,\gamma\rangle=
\langle w(\alpha_s),w(\alpha_t)\rangle=
\langle\alpha_s,\alpha_t\rangle=
-2\cos(\pi/m_{s,t})=
-2\cos(\pi/\hat m_{\beta,\gamma})=
\langle\!\langle\hat E_\beta,\hat E_\gamma\rangle\!\rangle\,.
\]
This completes the proof of Claim 2.

Now take a free of infinity finite subset $\XX$ of $\Phi[\Gamma]$.
We turn to show that $A[\hat\Gamma_\XX]$ is of spherical type or of affine type, that $|\XX|\le 2\,n_\sph(A[\Gamma])$, and that $n_\sph(A[\hat\Gamma_\XX])\le n_\sph(A[\Gamma])$. The existence of this uniform bound on the cardinality of $\XX$ in particular implies that there exists no free of infinity infinite subset of $\Phi[\Gamma]$.
This will prove the theorem.
We denote by $\hat V_\XX$ the linear subspace of $\hat V$ spanned by $\{\hat E_\beta \mid \beta\in\XX\}$ and by $\langle\!\langle\cdot,\cdot\rangle\!\rangle_\XX:\hat V_\XX\times\hat V_\XX\to \R$ the restriction of $\langle\!\langle\cdot,\cdot\rangle\!\rangle$ to $\hat V_\XX$.

Let $\hat u\in \hat V_\XX$.
We write $\hat u=\sum_{\beta\in\XX}\lambda_\beta\hat E_\beta$ with $\lambda_\beta\in\R$ for all $\beta\in\XX$.
Let $u=\sum_{\beta\in\XX}\lambda_\beta \beta \in V$.
Then, by Claim 2,
\[
\langle\!\langle\hat u,\hat u\rangle\!\rangle_\XX =
\sum_{\beta\in\XX}\sum_{\gamma\in\XX}\lambda_\beta\lambda_\gamma\langle\!\langle\hat E_\beta,\hat E_\gamma\rangle\!\rangle =
\sum_{\beta\in\XX}\sum_{\gamma\in\XX}\lambda_\beta\lambda_\gamma\langle\beta,\gamma\rangle=
\langle u,u\rangle\ge0\,.
\]
This shows that $\langle\!\langle\cdot,\cdot\rangle\!\rangle_\XX$ is positive, hence $A[\hat\Gamma_\XX]$ is of spherical type or of affine type. 

Let $\hat\Gamma_1,\dots,\hat\Gamma_m$ be the connected components of $\hat\Gamma_\XX$.
For each $k\in\{1,\dots,m\}$ we denote by $\XX_k$ the set of vertices of $\hat\Gamma_k$ and we consider the matrix $A_k=(\hat m_{\beta,\gamma})_{\beta,\gamma\in\XX_k}$.
The matrix $A_k$ is symmetric, positive and indecomposable, hence, by Claim 1, the rank of $A_k$ is either $|\XX_k|$ or $|\XX_k|-1$.
In the second case we necessarily have $|\XX_k|-1\ge 1$.
Thus, we have $m\le\rank(\langle\!\langle\cdot,\cdot\rangle\!\rangle_\XX)$, and therefore:
\[
|\XX|=
\sum_{k=1}^m |\XX_k| =
m+\sum_{k=1}^m(|\XX_k|-1) \le
m+\rank(\langle\!\langle\cdot,\cdot\rangle\!\rangle_\XX)\le
2\,\rank(\langle\!\langle\cdot,\cdot\rangle\!\rangle_\XX)\,.
\]
Since both $\hat\Gamma_\XX$ and $\Gamma$ are of spherical type or of affine type, we have $n_\sph(A[\hat\Gamma_\XX])=\rank(\langle\!\langle\cdot,\cdot\rangle\!\rangle_\XX)$ and $n_\sph(A[\Gamma])=\rank(\langle\cdot,\cdot\rangle)$, hence it remains to show that $\rank(\langle\!\langle\cdot,\cdot\rangle\!\rangle_\XX)\le\rank(\langle\cdot,\cdot\rangle)$.

By applying again Claim 1 to the matrices $A_k$ we see that there exists a subset $\YY\subset\XX$ such that $|\YY|=\rank(\langle\!\langle\cdot,\cdot\rangle\!\rangle_\XX)$ and $\langle\!\langle\cdot,\cdot\rangle\!\rangle_\YY$ is positive definite. 
Let $\{\lambda_\beta\}_{\beta\in\YY}$ be a collection of real numbers such that $\sum_{\beta\in\YY} \lambda_\beta \beta=0$ (in $V$).
Then, by Claim 2,
\begin{gather*}
0=\langle 0,0\rangle=
\langle\sum_{\beta\in\YY} \lambda_\beta \beta,\sum_{\beta\in\YY} \lambda_\beta \beta\rangle=
\sum_{\beta\in\YY}\sum_{\gamma\in\YY}\lambda_\beta\lambda_\gamma\langle\beta,\gamma\rangle=\\
\sum_{\beta\in\YY}\sum_{\gamma\in\YY}\lambda_\beta\lambda_\gamma\langle\!\langle\hat E_\beta,\hat E_\gamma\rangle\!\rangle=
\langle\!\langle\sum_{\beta\in\YY}\lambda_\beta\hat E_\beta,\sum_{\beta\in\YY}\lambda_\beta\hat E_\beta\rangle\!\rangle_\YY\,.
\end{gather*}
Since $\langle\!\langle\cdot,\cdot\rangle\!\rangle_\YY$ is positive definite and $\{\hat E_\beta\mid\beta\in\YY\}$ is a basis of $\hat V_\YY$, it follows that $\lambda_\beta=0$ for all $\beta\in\YY$, hence $\YY$ is a linearly independent subset of $V$.
Let $V_\YY$ be the linear subspace of $V$ spanned by $\YY$ and let $\langle\cdot,\cdot\rangle_\YY:V_\YY\times V_\YY\to\R$ be the restriction of $\langle\cdot,\cdot\rangle$ to $V_\YY$.
The matrix of $\langle\cdot,\cdot\rangle_\YY$ in the basis $\YY$ is the same as the matrix of $\langle\!\langle\cdot,\cdot\rangle\!\rangle_\YY$ in the basis $\{\hat E_\beta\mid\beta\in\YY\}$ and $\langle\!\langle\cdot,\cdot\rangle\!\rangle_\YY$ is positive definite, hence $\langle\cdot,\cdot\rangle_\YY$ is positive definite.
So, $\rank(\langle\!\langle\cdot,\cdot\rangle\!\rangle_\XX)=|\YY| = \dim(V_\YY)=\rank(\langle\cdot,\cdot\rangle_\YY) \le \rank(\langle\cdot,\cdot\rangle)$.
\end{proof}


\section{Word problem}\label{sec5}

\begin{thm}\label{thm5_1}
Let $\Gamma$ be a Coxeter graph of spherical type or of affine type.
Then $\VA[\Gamma]$ has a solution to the word problem.
\end{thm}

The proof of Theorem \ref{thm5_1} relies on the following proposition which we state and prove in a general framework, without assuming that the Coxeter graph is of spherical type or of affine type.

\begin{prop}\label{prop5_2}
Let $\Gamma$ be a finite Coxeter graph.
Let $\SS\sqcup\TT$ be the standard generating set of $\VA[\Gamma]$, where $\SS=\{\sigma_s\mid s\in S\}$ and $\TT=\{\tau_s\mid s\in S\}$.
There exists an algorithm which, given a word $w$ on $(\SS\sqcup\TT)^{\pm 1}$ which represents an element $g\in\KVA[\Gamma]$, determines:
\begin{itemize}
\item[(i)]
a finite subset $\XX\subset\Phi[\Gamma]$ such that $g\in A[\hat\Gamma_\XX]$,
\item[(ii)]
the Coxeter matrix $\hat M_\XX$ of $\hat\Gamma_\XX$,
\item[(iii)]
a word $\hat w$ over $\{\delta_\beta\mid\beta\in\XX\}^{\pm 1}$ which represents $g$.
\end{itemize}
\end{prop}

\begin{proof}
Let $M=(m_{s,t})_{s,t\in S}$ be the Coxeter matrix of $\Gamma$.
For $s,t\in S$ we set $a_{s,t}=-2\cos(\pi/m_{s,t})$ if $m_{s,t}\neq\infty$ and $a_{s,t}=-2$ if $m_{s,t}=\infty$.
The number $a_{s,t}$ is an algebraic number for all $s,t \in S$, hence we can consider a number field $\K$ containing all $a_{s,t}$ and make all linear calculations with $\K$ in place of $\R$.
This makes them algorithmically computable.
Take an abstract set $\Pi=\{\alpha_s\mid s\in S\}$ in one-to-one correspondence with $S$ and denote by $V$ the vector space over $\K$ having $\Pi$ as a basis.
Moreover, consider the symmetric bilinear form $\langle\cdot,\cdot\rangle:V\times V \to\K$ defined by $\langle\alpha_s,\alpha_t\rangle=a_{s,t}$ for all $s,t\in S$ and the embedding $W[\Gamma]\hookrightarrow\GL(V)$ defined by $s(v)= v-\langle v,\alpha_s\rangle\alpha_s$ for $s\in S$ and $v\in V$.
Set $\Phi[\Gamma]=\{w(\alpha_s)\mid s\in S\text{ and } w\in W[\Gamma]\}$.

In this context a root $\beta$ is given by a word $w=s_1\cdots s_\ell\in S^*$ and an element $s\in S$ such that $\bar w(\alpha_s) = \beta$, where $\bar w$ is the element of $W[\Gamma]$ represented by $w$.
The word problem in $W[\Gamma]$ is solvable (see Tits \cite{Tits1}) and the linear operations in $V$ can be down with algorithmic means, hence we can decide whether two pairs $(w,s)$ and $(w',s')$ in $S^*\times S$ represent the same root or not.

{\it Claim 1.}
Let $\beta,\gamma\in\Phi[\Gamma]$ such that $\langle\beta,\gamma\rangle\le0$ and let $s,t\in S$ such that $m_{s,t}\neq\infty$.
There exists $u\in W[\Gamma]$ such that $\{u(\alpha_s),u(\alpha_t)\}=\{\beta,\gamma\}$ if and only if $st$ and $r_\beta r_\gamma$ are conjugate in $W[\Gamma]$.

{\it Proof of Claim 1.}
Note that $ts$ and $st$ are conjugate, since $ts=s\,st\,s$, hence we can replace $st$ with $ts$ if necessary. 
Suppose first there exists $u\in W[\Gamma]$ such that $\{u(\alpha_s),u(\alpha_t)\}=\{\beta,\gamma\}$, say $u(\alpha_s)=\beta$ and $u(\alpha_t)=\gamma$.
Then $r_\beta = ur_{\alpha_s} u^{-1} = usu^{-1}$ and $r_\gamma=ur_{\alpha_t}u^{-1}=utu^{-1}$, hence $r_\beta r_\gamma = u\,st\,u^{-1}$.

Now, suppose that $r_\beta r_\gamma$ and $st$ are conjugate in $W[\Gamma]$.
Let $w\in W[\Gamma]$ such that $wr_\beta r_\gamma w^{-1} = st$.
Upon replacing $\beta$ with $w(\beta)$ and $\gamma$ with $w(\gamma)$, we can assume that $r_\beta r_\gamma = st$.
By Baumeister--Dyer--Stump--Wegener \cite[Theorem 1.4]{BaDyStWe1}, $r_\beta$ and $r_\gamma$ lie in the subgroup $W[\Gamma_{\{s,t\}}]$ of $W[\Gamma]$ generated by $\{s,t\}$, and they are conjugate in $W[\Gamma_{\{s,t\}}]$ to elements of $\{s,t\}$.
Since, furthermore, $r_\beta r_\gamma = st$, there exists $k\in\N$ such that $r_\beta=(st)^{k+1}s$ and $r_\gamma=(st)^ks$.
If $k$ is even, say $k=2\ell$, then, upon conjugating by $s(ts)^{\ell}$, we may assume that $r_\beta=t$ and $r_\gamma=s$.
If $k$ is odd, say $k=2\ell-1$, then, upon conjugating by $(ts)^\ell$, we may assume that $r_\beta=s$ and $r_\gamma=t$.
So, since we can swap $s$ and $t$, we may assume that $r_\beta=s$ and $r_\gamma=t$.
This means in terms of roots that $\beta \in\{\alpha_s,-\alpha_s\}$ and $\gamma \in\{\alpha_t,-\alpha_t\}$.
If $\langle\beta,\gamma\rangle=0$, then $\langle\alpha_s,\alpha_t\rangle=0$.
In this case $s(\alpha_s)=-\alpha_s$, $s(\alpha_t)=\alpha_t$, $t(\alpha_s)=\alpha_s$ and $t(\alpha_t)=-\alpha_t$.
By using these equalities we easily reduce to the case where $\beta=\alpha_s$ and $\gamma=\alpha_t$, which concludes the proof in this case.
Suppose $\langle \beta,\gamma\rangle<0$.
Then $m_{s,t}\ge 3$, $\langle\alpha_s,\alpha_t\rangle=\langle\beta,\gamma\rangle$, and we have two possibilities: either $\beta=\alpha_s$ and $\gamma=\alpha_t$, or $\beta=-\alpha_s$ and $\gamma=-\alpha_t$.
If $\beta=\alpha_s$ and $\gamma=\alpha_t$, then there is nothing to prove. 
So, we can assume $\beta=-\alpha_s$ and $\gamma=-\alpha_t$.
We know by Deodhar \cite{Deodh1} that:
\begin{gather*}
\Prod_R(s,t,m_{s,t})(\alpha_s)=-\alpha_s \text{ and }\Prod_R(s,t,m_{s,t})(\alpha_t)=-\alpha_t\text{ if }m_{s,t}\text{ is even}\,,\\
\Prod_R(s,t,m_{s,t})(\alpha_s)=-\alpha_t \text{ and }\Prod_R(s,t,m_{s,t})(\alpha_t)=-\alpha_s\text{ if }m_{s,t}\text{ is odd}\,,
\end{gather*}
hence, upon applying $\Prod_R(s,t,m_{s,t})$, we may assume that $\beta=\alpha_s$ and $\gamma=\alpha_t$, which finishes the proof of Claim 1.

{\it Claim 2.}
There exists an algorithm which, for given $\beta,\gamma\in\Phi[\Gamma]$, computes $\hat m_{\beta,\gamma}$.

{\it Proof of Claim 2.}
If $\langle\beta,\gamma\rangle\not\in]-2,0]$, then $\hat m_{\beta,\gamma}=\infty$.
Thus, we can assume $\langle\beta,\gamma\rangle\in]-2,0]$.
As shown in the beginning of the proof of Proposition 1.4.7 in Krammer \cite{Kramm1}, this implies that $r_\beta r_\gamma$ is of finite order.
Let $s,t\in S$ such that $m_{s,t}\neq\infty$.
Then the algorithm of Krammer \cite[Proposition 3.3.2]{Kramm1} determines whether $r_\beta r_\gamma$ and $st$ are conjugate. 
By Claim 1 this algorithm also decides whether there exists $u\in W[\Gamma]$ such that $\{u(\alpha_s),u(\alpha_t)\}=\{\beta,\gamma\}$. 
Then we compute $\hat m_{\beta,\gamma}$ by applying this algorithm to all the pairs $(s,t)\in S\times S$ such that $s\neq t$ and $m_{s,t}\neq\infty$.
This concludes the proof of Claim 2.

{\it End of the proof of Proposition \ref{prop5_2}.}
Let $g\in\KVA[\Gamma]$ given by a word over $(\SS\sqcup\TT)^{\pm 1}$.
From this word we can easily write $g$ in the form: 
\[
g=\iota_W(w_0)\sigma_{s_1}^{\varepsilon_1}\iota_W(w_1)\cdots\sigma_{s_p}^{\varepsilon_p}\iota_W(w_p)\,,
\]
where $w_0,w_1,\dots,w_p\in W[\Gamma]$, $s_1,\dots,s_p\in S$ and $\varepsilon_1,\dots,\varepsilon_p\in\{\pm 1\}$.
For $i=1,\dots,p$ set $\beta_i=(w_0\cdots w_{i-1})(\alpha_{s_i})$.
Set also $w=w_0w_1\cdots w_p$.
Then $g=\delta_{\beta_1}^{\varepsilon_1}\delta_{\beta_2}^{\varepsilon_2}\cdots\delta_{\beta_p}^{\varepsilon_p}\iota_W(w)$.
But $w=\pi_K(g)=1$, since $g\in\KVA[\Gamma]$, hence $g=\delta_{\beta_1}^{\varepsilon_1}\delta_{\beta_2}^{\varepsilon_2}\cdots\delta_{\beta_p}^{\varepsilon_p}$.
This proves Part (i) and Part (iii) if we set $\XX=\{\beta_1,\dots,\beta_p\}$.
Then Part (ii) immediately follows from Claim 2.
\end{proof}

\begin{proof}[Proof of Theorem \ref{thm5_1}]
Let $\Gamma$ be a Coxeter graph of spherical type or of affine type. 
Let $g\in\VA[\Gamma]$ given by a word over $(\SS\sqcup\TT)^{\pm 1}$.
We can calculate $\pi_K(g)\in W[\Gamma]$ and the word problem in $W[\Gamma]$ is solvable (see Tits \cite{Tits1}), hence we can decide whether $\pi_K(g)$ is trivial or not. 
If $\pi_K(g)\neq 1$, then $g\neq 1$ and the algorithm stops here. 
So, we can assume $\pi_K(g)=1$, that is, $g\in\KVA[\Gamma]$.
The algorithm of Proposition \ref{prop5_2} gives a finite subset $\XX\subset\Phi[\Gamma]$ and a word $\hat w$ over $\{\delta_\beta\mid\beta\in\XX\}^{\pm 1}$ such that $g\in A[\hat\Gamma_\XX]$ and $\hat w$ represents $g$.
It also gives the matrix $\hat M_\XX$ of $\hat\Gamma_\XX$.
We know by Godelle--Paris \cite{GodPar1} that, if, for each free of infinity subset $\YY\subset\XX$ the group $A[\hat\Gamma_\YY]$ has a solution to the word problem, then $A[\hat\Gamma_\XX]$ itself has a solution to the word problem. 
On the other hand, by Theorem \ref{thm4_1}, $A[\hat\Gamma_\YY]$ is of spherical type or of affine type for each free of infinity subset $\YY\subset\XX$, by Deligne \cite{Delig1} and Brieskorn--Saito \cite{BriSai1} the Artin groups of spherical type have a solvable word problem, and by McCammond--Sulway \cite{McCSul1} the Artin groups of affine type have a solvable word problem.
So, $A[\hat\Gamma_\XX]$ has a solution to the word problem.
This means that there is an algorithm which decides whether $g$ is trivial or not.
\end{proof}


\section{$K(\pi,1)$ conjecture and cohomological dimension}\label{sec6}

We start by recalling the $K(\pi,1)$ conjecture for Artin groups. 
Let $\Gamma$ be a Coxeter graph and let $M=(m_{s,t})_{s,t\in S}$ be its Coxeter matrix.
Recall that $\Pi=\{\alpha_s\mid s\in S\}$ is the set of simple roots, that $V$ is the real vector space having $\Pi$ as a basis, and that $\langle\cdot,\cdot\rangle :V\times V\to\R$ is the canonical bilinear form. 
We have an embedding $W[\Gamma]\hookrightarrow\GL(V)$ defined by $s(v)=v-\langle v,\alpha_s\rangle\alpha_s$ for $s\in S$ and $v\in V$, and $W[\Gamma]$, viewed as a subgroup of $\GL(V)$, is generated by linear reflections.

Suppose $\Gamma$ is finite.
There exists a non-empty open cone $I\subset V$, called \emph{Tits cone}, on which $W[\Gamma]$ acts properly discontinuously.
The set of reflections in $W[\Gamma]$ is $R=\{wsw^{-1}\mid w\in W[\Gamma]\text{ and } s\in S\}$ and $W[\Gamma]$ acts freely and properly discontinuously on $I\setminus\cup_{r\in R} H_r$, where, for $r\in R$, $H_r$ denotes the hyperplane of $V$ fixed by $r$ (see Bourbaki \cite{Bourb1}).
Set:
\[
\MM[\Gamma]=(I\times I)\setminus\left(\bigcup_{r\in R} (H_r\times H_r)\right)\,.
\]
This is a connected manifold of dimension $2\,|S|$ on which $W[\Gamma]$ acts freely and properly discontinuously. 
A key result in the domain is the following.

\begin{thm}[Van der Lek \cite{Lek1}]\label{thm6_1}
Let $\Gamma$ be a finite Coxeter graph.
Then the fundamental group of $\MM[\Gamma]/W[\Gamma]$ is isomorphic to $A[\Gamma]$.
\end{thm}

Recall that a CW-complex $X$ is a \emph{classifying space} for a (discrete) group $G$ if $\pi_1(X)=G$ and the universal covering of $X$ is contractible. 
These spaces play a major role in the computation of cohomology of groups (see Brown \cite{Brown1}).
The following is one of the main questions in the theory of Artin groups.

\begin{conj}[$K(\pi,1)$ Conjecture]\label{conj6_2}
Let $\Gamma$ be a finite Coxeter graph.
Then $\MM[\Gamma]/W[\Gamma]$ is a classifying space for $A[\Gamma]$.
\end{conj}

We say that an Artin group $A[\Gamma]$ with $\Gamma$ finite \emph{satisfies  the $K(\pi,1)$ conjecture} if $\MM[\Gamma]/W[\Gamma]$ is a classifying space for $A[\Gamma]$.
So, the above conjecture says that any Artin group $A[\Gamma]$ with $\Gamma$ finite satisfies the $K(\pi,1)$ conjecture.
Note also that, by Deligne \cite{Delig1} and Paolini--Salvetti \cite{PaoSal1}, $A[\Gamma]$ satisfies the $K(\pi,1)$ conjecture if $\Gamma$ is of spherical type or of affine type.

The following makes sense only when $\Gamma$ is of spherical type since this is the only case where $\hat\Gamma$ is finite. 

\begin{thm}\label{thm6_3}
Let $\Gamma$ be a Coxeter graph of spherical type.
Then $A[\hat\Gamma]$ satisfies the $K(\pi,1)$ conjecture.
\end{thm}

Theorem \ref{thm6_3} is a direct consequence of the following proposition combined with Theorem \ref{thm4_1}.

\begin{prop}\label{prop6_4}
Let $\Gamma$ be a finite Coxeter graph and let $M=(m_{s,t})_{s,t\in S}$ be its Coxeter matrix.
Suppose that, for each free of infinity subset $X\subset S$, the Coxeter graph $\Gamma_X$ is of spherical type or of affine type.
Then $A[\Gamma]$ satisfies the $K(\pi,1)$ conjecture.
\end{prop}

\begin{proof}
By Deligne \cite{Delig1}, if $\Gamma$ is of spherical type, then $A[\Gamma]$ satisfies the $K(\pi,1)$ conjecture, and by Paolini--Salvetti \cite{PaoSal1}, if $\Gamma$ is of affine type, then $A[\Gamma]$ satisfies the $K(\pi,1)$ conjecture.
On the other hand, by Ellis--Sk\"oldberg \cite{EllSko1} (see also Godelle--Paris \cite{GodPar2}), if $A[\Gamma_X]$ satisfies the $K(\pi,1)$ conjecture for every free of infinity subset $X\subset S$, then $A[\Gamma]$ satisfies the $K(\pi,1)$ conjecture.
So, if for each free of infinity subset $X\subset S$ the Coxeter graph $\Gamma_X$ is of spherical type or of affine type, then $A[\Gamma]$ satisfies the $K(\pi,1)$ conjecture.
\end{proof}

\begin{proof}[Proof of Theorem \ref{thm6_3}]
It is a straightforward consequence of Proposition \ref{prop6_4} and Theorem \ref{thm4_1}.
\end{proof}

For a group $G$ we denote by $\cd(G)$ the cohomological dimension of $G$ and by $\vcd(G)$ the virtual cohomological dimension of $G$.

\begin{corl}\label{corl6_5}
Let $\Gamma$ be a Coxeter graph of spherical type and let $n$ be the number of vertices of $\Gamma$.
Then $\cd(\KVA[\Gamma])=\vcd(\VA[\Gamma])=n$.
In particular, $\KVA[\Gamma]$ is torsion free and $\VA[\Gamma]$ is virtually torsion free.
\end{corl}

\begin{proof}
We have $n=n_\sph(A[\Gamma])$ by definition.
By Theorem \ref{thm6_3}, $A[\hat\Gamma]$ satisfies the $K(\pi,1)$ conjecture, hence, by Charney--Davis \cite{ChaDav1}, $\cd(\KVA[\Gamma])=\cd(A[\hat\Gamma])=n_\sph(A[\hat\Gamma])$.
By Theorem \ref{thm4_1} it follows that $\cd(\KVA[\Gamma])\le n$.
On the other hand, $A[\Gamma]$ is a subgroup of $\KVA[\Gamma]$ (see Corollary \ref{corl2_4}) and, by Deligne \cite{Delig1}, $\cd(A[\Gamma])=n$, hence $n\le \cd(\KVA[\Gamma])$, and therefore $\cd(\KVA[\Gamma])=n$.
Finally, since $\VA[\Gamma]=\KVA[\Gamma]\rtimes W[\Gamma]$ and $W[\Gamma]$ is finite, $\vcd(\VA[\Gamma])=\cd(\KVA[\Gamma])=n$.
\end{proof}

\begin{rem}
It follows from Bartholdi--Enriquez--Etingof--Rains \cite{BaEnEtRa1} that the cohomological dimension of $\PVB_n$ is $n-1$, which implies also that $\PVB_n$ is torsion free. 
We think that this extends to all virtual Artin groups of spherical type in the sense that, for a Coxeter graph $\Gamma$ of spherical type, we think that the cohomological dimension of $\PVA[\Gamma]$ is equal to the number of vertices of $\Gamma$, and therefore that $\PVA[\Gamma]$ is torsion free as well. 
Note that, since in this case $\PVA[\Gamma]$ is of finite index in $\VA[\Gamma]$ and by Corollary \ref{corl6_5} the virtual cohomological dimension of $\VA[\Gamma]$ is equal to the number of vertices of $\Gamma$, by a theorem of Serre \cite{Serre1} it would suffice to prove that $\PVA[\Gamma]$ is torsion free.
\end{rem}

We are not able to give precise values for the cohomological dimension of $\KVA[\Gamma]$ and for the virtual cohomological dimension of $\VA[\Gamma]$ when $\Gamma$ is of affine type, but we can give bounds:

\begin{thm}\label{thm6_6}
Let $\Gamma$ be a Coxeter graph of affine type. 
\begin{itemize}
\item[(1)]
$\KVA[\Gamma]$ is torsion free and $\cd(\KVA([\Gamma]) \le n_\sph(A[\Gamma])+1$.
\item[(2)]
$\VA[\Gamma]$ is virtually torsion free and $\vcd(\VA[\Gamma])\le 2\,n_\sph(A[\Gamma])+1$.
\end{itemize}
\end{thm}

\begin{proof}
Let $\XX$ be a finite subset of $\Phi[\Gamma]$.
We denote by $\PP_{<\infty}(\XX)$ the set of free of infinity subsets of $\XX$.
Recall that $n_\sph(A[\hat\Gamma_\XX])=\max\{n_\sph(A[\hat\Gamma_\YY])\mid\YY\in\PP_{<\infty}(\XX)\}$ and, by Theorem \ref{thm4_1}, $n_\sph(A[\hat\Gamma_\YY])\le n_\sph(A[\Gamma])$ for all $\YY\in\PP_{<\infty}(\XX)$, hence $n_\sph(A[\hat\Gamma_\XX])\le n_\sph(A[\Gamma])$.
Furthermore, by Proposition \ref{prop6_4} and Theorem \ref{thm4_1}, $A[\hat\Gamma_\XX]$ satisfies the $K(\pi,1)$ conjecture, hence, by Charney--Davis \cite{ChaDav1}, $\cd(A[\hat\Gamma_\XX])=n_\sph(A[\hat\Gamma_\XX])\le n_\sph(A[\Gamma])$.
Recall that $\PP_\fin(\Phi[\Gamma])$ denotes the set of finite subsets of $\Phi[\Gamma]$.
Since $\KVA[\Gamma]=A[\hat\Gamma]=\varinjlim_{\XX\in\PP_\fin(\Phi[\Gamma])} A[\hat\Gamma_\XX]$ and $\cd(A[\hat\Gamma_\XX]) \le n_\sph(A[\Gamma])$ for all $\XX\in\PP_\fin(\Phi[\Gamma])$, by Bieri \cite[Theorem 4.7]{Bieri1} we conclude that $\cd(\KVA[\Gamma])\le n_\sph(A[\Gamma])+1$.
This also implies that $\KVA[\Gamma]$ is torsion free.

We know by Bourbaki \cite[Chapitre VI]{Bourb1} that $W[\Gamma]$ admits a semi-direct product decomposition $W[\Gamma]=\Z^m\rtimes W_0$ where $m\le n_\sph(A[\Gamma])$ and $W_0$ is a finite Coxeter group.
Let $H=\pi_K^{-1}(\Z^m)$.
Then $H$ is a finite index subgroup of $\VA[\Gamma]$ which admits a semi-direct product decomposition $H=\KVA[\Gamma]\rtimes\Z^m$.
By Brown \cite[Proposition 2.4]{Brown1} it follows that:
\[
\cd(H)\le 
\cd(\KVA[\Gamma])+\cd(\Z^m) \le
n_\sph(A[\Gamma])+1+m\le
2\,n_\sph(A[\Gamma])+1\,,
\]
hence $\VA[\Gamma]$ is virtually torsion free and $\vcd(\VA[\Gamma])=\cd(H)\le 2\,n_\sph(A[\Gamma])+1$.
\end{proof}



\end{document}